\documentclass{amsart}

\usepackage{amsfonts, amsmath, amscd}
\usepackage[psamsfonts]{amssymb}
\usepackage{amssymb}
\usepackage{pb-diagram}
\usepackage[usenames]{color}
\usepackage{hyperref}
\usepackage[all]{xy}
\usepackage{graphicx}
\numberwithin{equation}{section}

\newtheorem{theorem}{Theorem}[section]
\newtheorem{corollary}[theorem]{Corollary}
\newtheorem{proposition}[theorem]{Proposition}
\newtheorem{lemma}[theorem]{Lemma}
\newtheorem{remark}[theorem]{Remark}
\theoremstyle{definition}

\newtheorem{definition}[theorem]{Definition}

\newtheorem{problem}{Problem}

\theoremstyle{plain}

\newcommand{\R}{\mathbb{R}}

\def\N{\mathbb N}
\def\R{\mathbb R}

\def\emp{\emptyset}

\begin{document}

\title[Lattices of homomorphisms and pro-Lie groups]{Lattices of homomorphisms and pro-Lie groups}
\author[Arkady G.~Leiderman and Mikhail G.~Tkachenko]{Arkady G.~Leiderman and Mikhail G.~Tkachenko$\,^1$}
\address{Department of Mathematics, Ben-Gurion University of the Negev, Beer Sheva, P.O.B. 653, Israel}
\email{arkady@math.bgu.ac.il}
\address{Departamento de Matem\'aticas, Universidad Aut\'onoma Metropolitana, Av. San Rafael Atlixco 186, Col. Vicentina, Del. Iztapalapa, C.P. 09340, Ciudad de M\'exico, Mexico}
\email{mich@xanum.uam.mx}
\thanks{$^1\,$The second listed author gratefully acknowledges the financial support received from the Center for Advanced Studies in Mathematics of the Ben Gurion University of the Negev during his 
visit in October, 2015. He was also supported by the Consejo Nacional de Ciencia y Tecnolog\'{\i}a (CONACyT) of Mexico, grant number 265992 (Estancias Sab\'aticas en el Extranjero).}


\date{\today}

\begin{abstract}
Early this century K.\,H.~Hofmann and S.\,A.~Morris introduced the class of \textit{pro-Lie groups} which consists of projective limits of finite-dimensional Lie groups and proved that it contains all compact groups, all locally compact abelian groups, and all connected locally compact groups and is closed under the formation of products and closed subgroups. They defined a topological group $G$ to be \textit{almost connected} if the quotient group of $G$ by the connected component of its identity is compact.

We show here that all almost connected pro-Lie groups as well as their continuous homomorphic images are $\R$-factorizable and \textit{$\omega$-cellular}, i.e.~every family of $G_\delta$-sets contains a countable subfamily whose union is dense in the union of the whole family.
We also prove a general result which implies as a special case that if a topological group $G$ contains a compact invariant subgroup $K$ such that the quotient group $G/K$ is an almost connected pro-Lie group, then $G$ is $\R$-factorizable and $\omega$-cellular.

Applying the aforementioned result we show that the sequential closure and the closure of an 
arbitrary $G_{\delta,\Sigma}$-set in an almost connected pro-Lie group $H$ coincide.
\end{abstract}

\maketitle

%%%%%%%%%%%%%%%%%%%%%%%%%%%%
\section{Introduction}%\label{intro}
A topological group is called a \emph{pro-Lie group} \cite{PROBOOK, HM2,HM3} if it is 
a projective limit of finite-dimensional Lie groups. As shown in \cite{PROBOOK} the class of  
pro-Lie groups includes all locally compact abelian topological groups, all compact groups, 
all connected locally compact topological groups, and all almost connected locally compact 
topological groups. Further, the class of pro-Lie groups is productive and every closed 
subgroup of a pro-Lie group is again a pro-Lie group.

Our main objective is to study the topological properties of \emph{almost connected} 
pro-Lie groups, i.e.~the pro-Lie groups $G$ such that the quotient group $G/G_0$ is compact, 
where $G_0$ is the connected component of $G$. The following theorem about the topological structure of almost connected pro-Lie groups plays a key role throughout the paper. According 
to \cite[Corollary~8.9]{HM2}, an almost connected pro-Lie group is \emph{homeomorphic} to 
the product $\R^\kappa\times K$, where $\R$ is the real line with the usual topology, $\kappa$ 
is a cardinal, and $K$ is a compact topological group.

A topological group $G$ is said to be \emph{$\R$-factorizable} if, for any continuous function 
$f\colon G\to \R$, one can find a second-countable topological group $S$, a continuous homomorphism $\pi\colon G\to S$, and a continuous function $h\colon S\to \R$ for which 
$f=h\circ \pi$.

Evidently, every second-countable group is $\R$-factorizable. In fact, an arbitrary topological 
product of second-countable topological groups is $\R$-factorizable \cite[Corollary~8.1.15]{AT}. 
A theorem of Pontryagin, reformulated in modern terms, asserts that every compact 
topological group is $\R$-factorizable (see \cite[Section~8.1]{AT}).

$\R$-factorizable groups were introduced in \cite{Tk91}. The class of $\R$-factorizable 
groups is sufficiently large\,---\,it contains arbitrary subgroups of Lindel\"of $\Sigma$-groups, 
direct products of Lindel\"of $\Sigma$-groups, and their dense subgroups (see 
\cite[Section~8.1]{AT}). In particular, every subgroup of a $\sigma$-compact group is 
$\R$-factorizable. Nevertheless, several major questions are still open. It is not known 
whether the class of $\R$-factorizable groups is closed under taking direct products, or 
whether this class is closed under passing to continuous homomorphic images.

Delimiting the frontiers of the class of $\R$-factorizable groups is another important 
part of research. It is known that \emph{$\omega$-narrow} topological groups can
fail to be $\R$-factorizable \cite[Example~8.2.1]{AT}. Recently, E.~Reznichenko and 
O.~Sipacheva \cite{RS} obtained a considerably stronger result. They showed that 
there exists a separable (and hence $ccc$ and $\omega$-narrow) topological group 
which is not $\R$-factorizable.

Thus, the question about $\R$-factorizability of almost connected pro-Lie groups and 
close to them topological groups arises quite naturally.

Section~\ref{R_pro-Lie} of the article is devoted to the study of continuous homomorphic 
images of the almost connected pro-Lie groups. It is shown in Theorem~\ref{Th:1} that every continuous homomorphic image $H$ of an almost connected pro-Lie group, and the 
Hewitt-Nachbin completion of $H$, are both $\R$-factorizable and $\omega$-cellular. Further, Theorem~\ref{Th:1y} states that the group $H$ is an Efimov space, and that the Hewitt-Nachbin completion of $H$ remains to be Efimov spaces as well. All statements mentioned above fail 
to hold for general pro-Lie groups, without the assumption of almost connectedness (see 
Remark~\ref{Rem:1}). To show that almost connected pro-Lie groups have the Efimov 
property we make use of the notion of \emph{weak $\sigma$-lattice} introduced in \cite{Tk94}.

In Section~\ref{Strong_lattices} we further develop a technique of strong $\sigma$-lattices 
of continuous mappings (or homomorphisms) which plays an important role in the subsequent Sections~\ref{Ext_pro-Lie} and~\ref{Sec:CS}. Several of our results in Section~\ref{Strong_lattices} have value in their own right. For example, we show in Theorem~\ref{Th:Ext} that if $K$ is a 
compact invariant subgroup of a topological group $G$ and the quotient group $G/K$ has a 
strong $\sigma$-lattice of open homomorphisms onto topological groups with a countable 
network (countable base), then $G$ also has such a lattice of open homomorphisms.

In Section~\ref{Ext_pro-Lie} we consider several topological properties of 
extensions of the almost connected pro-Lie groups. With the help of results of 
Section~\ref{Strong_lattices} we extend the conclusions of Theorems~\ref{Th:1} 
and~\ref{Th:1y} to topological groups $H$ which contain a compact invariant subgroup 
$K$ such that the quotient group $H/K$ is \emph{homeomorphic} to the product 
$C\times\prod_{i\in I} H_i$, where $C$ is a compact group and each $H_i$ is
a topological group with a countable network (see Theorem~\ref{Th:2} ).

Convergence properties of almost connected pro-Lie groups are considered in 
Section~\ref{Sec:CS}. We prove in Theorem~\ref{Th:4} that if $H$ is an almost
connected pro-Lie group and $P$ is a union of $G_\delta$-sets in $H$, then the
closure of $P$ and the sequential closure of $P$ in $H$ coincide. In other words,
for every $x\in\overline{P}$, the set $P$ contains a sequence converging to $x$.

%%%%%%%%%%%%%%%%%%%%%%%%%%%%%%%%%%
\subsection{Notation and terminology}
%%%%%%%%%%%%%%%%%%%%%%%%%%%%%%%%%%
A topological group $G$ is \textit{$\omega$-narrow} if it can be covered by countably 
many translates of an arbitrary neighborhood of the identity. 

A topological space $X$ is said to be \textit{$\omega$-cellular} or, in symbols, 
$cel_\omega(X)\leq\omega$ if every family $\gamma$ of $G_\delta$-sets in $X$ 
contains a countable subfamily $\lambda$ such that $\bigcup\lambda$ is dense in 
$\bigcup\gamma$. It is clear that every $\omega$-cellular space has countable 
cellularity. The class of $\omega$-cellular spaces is considerably narrower than 
the class of spaces of countable cellularity\,---\,a space $X$ of countable pseudocharacter 
satisfies $cel_\omega(X)\leq\omega$ if and only if it is hereditarily separable.

The union of a family of $G_\delta$-sets in a space $X$ is said to be a $G_{\delta,\Sigma}$-set.
We also recall that a topological space $X$ is an \emph{Efimov} space if the closure 
in $X$ of every $G_{\delta,\Sigma}$-set is a $G_\delta$-set. A subset $Z$ of $X$ is a 
\emph{zero-set} if there exists a continuous real-valued function $f$ on $X$ such that 
$Z=f^{-1}(0)$. Clearly every closed $G_\delta$-set in a normal space is a zero-set.

The Hewitt--Nachbin completion (or \emph{realcompactification}) of a Tychonoff space $X$, 
denoted by $\upsilon{X}$, is a realcompact space containing a dense homeomorphic copy 
of $X$, such that $X$ is $C$-embedded in $\upsilon{X}$ (see \cite[Section~3.11]{Eng89}).

A Hausdorff space $X$ is a \emph{Lindel\"of $\Sigma$-space} if $X$ is a continuous image of a Hausdorff space $Y$ which admits a perfect mapping onto a separable metrizable space (for basic properties of Lindel\"of $\Sigma$-spaces see \cite[Section~5.3]{AT}).

A topological space $X$ is \emph{weakly Lindel\"of} if every open cover of $X$ contains a countable subfamily which covers a dense subset of $X$.  

A topological space $X$ has a \emph{$G_\delta$-diagonal} if the diagonal $\Delta_X=\{(x,x): x\in X\}$ is a $G_\delta$-set in $X^2$. As usual we say that a space $X$ is \emph{submetrizable} if it admits a coarser metrizable topology. It is clear that every submetrizable space $X$ is Hausdorff and has a \emph{regular} $G_\delta$-diagonal, i.e.~the diagonal $\Delta_X$ in $X^2$ is the intersection of countably many of its closed neighborhoods.

All topological spaces and topological groups are assumed to be Hausdorff.

%%%%%%%%%%%%%%%%%%%%%%%%%%%%%%%%%%%%%%%%% 
\section{Continuous homomorphic images of pro-Lie groups}\label{R_pro-Lie}
%%%%%%%%%%%%%%%%%%%%%%%%%%%%%%%%%%%%%%%%%
In this section we study the properties of continuous homomorphic images of \emph{almost connected} pro-Lie groups. It is worth mentioning that even quotient groups of pro-Lie 
groups may fail to be pro-Lie groups \cite[Corollary 4.11]{PROBOOK}

Let us recall that a \emph{paratopological} group is a group with a topology such that
multiplication on the group is continuous (but inversion can be discontinuous). The reader
can consult \cite[Section~2.3]{AT} for information about paratopological groups. We use 
this concept only in the proof of the next theorem. 

\begin{theorem}\label{Th:1}
Let a topological group $H$ be a continuous homomorphic image of an almost connected pro-Lie group $G$. Then the following hold:
\begin{enumerate}
\item[{\rm (a)}] The group $H$ is $\R$-factorizable and $\omega$-cellular.
\item[{\rm (b)}] The Hewitt--Nachbin completion of $H$, $\upsilon{H}$, is again an $\R$-factorizable and $\omega$-cellular topological group containing $H$ as a topological subgroup.
\end{enumerate} 
\end{theorem}

\begin{proof}
According to \cite[Corollary~8.9]{HM2}, the group $G$ is \emph{homeomorphic} to the product $\R^\kappa\times K$, where $\R$ is the real line with the usual topology, $\kappa$ is a cardinal, and $K$ is a compact topological group. Hence $H$ is a continuous image of the topological product $\R^\kappa\times K$, where the factors $\R$ (repeated $\kappa$ times) and $K$ are clearly Lindel\"of $\Sigma$-spaces. Therefore, combining Theorems~12 and~13 of \cite{Tk15} we conclude that the group $H$ is $\R$-factorizable and $\omega$-cellular and that the Hewitt--Nachbin completion of $H$ is an $\R$-factorizable paratopological group containing $H$ as a dense subgroup. Since $H$ is dense in $\upsilon{H}$, it follows from \cite[Proposition~2.2]{San_inf} that $\upsilon{H}$ is in fact a topological group. Thus $\upsilon{H}$ is an $\R$-factorizable topological group. 

It is easy to see that the space $\upsilon{H}$ is $\omega$-cellular. Indeed, let $\gamma$ be a family of $G_\delta$-sets in $\upsilon{H}$. Since $H$ is $G_\delta$-dense in $\upsilon{H}$, the set $P\cap H$ is dense in $P$, for each $P\in\gamma$. Hence there exists a countable subfamily $\lambda$ of $\gamma$ such that $H\cap\bigcup\lambda$ is dense in $H\cap\bigcup\gamma$. Then $\bigcup\lambda$ is dense in $\bigcup\gamma$, i.e.~the space $\upsilon{H}$ is $\omega$-cellular.
\end{proof}

\begin{corollary}\label{Cor:AlC}
Every almost connected pro-Lie group $H$ is $\R$-factorizable and $\omega$-cellular.
\end{corollary}

\begin{remark}\label{Rem:1}
{\rm One cannot drop \lq\lq{almost connected\rq\rq} in Theorem~\ref{Th:1}. Indeed, there exists a
prodiscrete group (that is, a complete group with a base of open subgroups) topologically isomorphic 
to a closed subgroup of the product of countable discrete Abelian groups which fails to be
$\R$-factorizable --- the group $H$ in \cite[Theorem~2.4]{Tk01} is as required. In particular, $H$ is 
an $\omega$-narrow pro-Lie group. Further, every $G_\delta$-set in $H$ is open, so $H$ is an uncountable \emph{$P$-group.} Hence $H$ is not $\omega$-cellular. In other words, all the conclusions of Theorem~\ref{Th:1} are false for $\omega$-narrow pro-Lie groups.}
\end{remark}

Our next aim in this section is to show that a continuous homomorphic image of an almost 
connected pro-Lie group is an Efimov space. For this purpose we make use the notion of a \emph{weak $\sigma$-lattice} and its modifications. 

Let us start with a definition of a partial preorder relation on continuous mappings of a given space. 

\begin{definition}\label{Def:1}
Let $f\colon X\to Y$ and $g\colon X\to Z$ be continuous onto mappings of topological spaces. 
We write $f\prec g$ if there exists a continuous mapping $h\colon Y\to Z$ satisfying $g=h\circ f$.
\[
\xymatrix{X\ar@{>}[r]^{f}\ar@{>}[dr]_{g} & Y
\ar@{>}[d]^{h}\\
&  Z }
\]
\end{definition}

Consider a family $\mathcal{L}$ of continuous mappings of a space $X$ elsewhere. We say that 
$\mathcal{L}$ is a \emph{weak $\sigma$-lattice} for $X$ if the following conditions are fulfilled:
\begin{enumerate}
\item[(1)] $\mathcal{L}$ generates the original topology of $X$, i.e.~$\mathcal{L}$ separates points and closed subsets of $X$;
\item[(2)] every finite subfamily of $\mathcal{L}$ has a lower bound in $(\mathcal{L},\prec);$
\item[(3)] for every decreasing sequence $p_0\succ p_1\succ p_2\succ \cdots$ in $\mathcal{L}$, there exists $p\in\mathcal{L}$ and a continuous one-to-one mapping $\phi\colon p(X)\to q(X)$ such that $q=\phi\circ{p}$, where $q$ is the diagonal product of the family $\{p_n: n\in\omega\}$.
\end{enumerate}
To visualize the notion of a weak $\sigma$-lattice, one can take a topological group $H$ and consider the family of all quotient mappings $\pi_N\colon H\to H/N$ onto left coset spaces $H/N$, where $N$ is an arbitrary closed subgroup of type $G_\delta$ in $H$. 

We will call a subgroup $N$ of a topological group $H$ \emph{admissible} (see \cite[Section~5.5]{AT})
if there exists a sequence $\{U_n: n\in\omega\}$ of open symmetric neighborhoods of the identity in $H$ such that $U_{n+1}^3\subset U_n$ for each $n\in\omega$ and $N=\bigcap_{n\in\omega} U_n$. It is easy to see that every admissible subgroup of $H$ is closed and the intersection of countably many admissible subgroups of $H$ is again an admissible  subgroup of $H$. 

\begin{lemma}\label{Le:1}
Let $H$ be a topological group and $\mathcal{N}$ the family of admissible subgroups of $H$. For every
$N\in\mathcal{N}$, let $\pi_N\colon H\to H/N$ be the quotient mapping onto the corresponding left coset space. Then the family $\mathcal{L}=\{\pi_N: N\in\mathcal{N}\}$ is a weak $\sigma$-lattice for $H$ which consists of open mappings onto submetrizable spaces. In particular, $H/N$ has a regular $G_\delta$-diagonal for each $N\in\mathcal{N}$.
\end{lemma}

\begin{proof}
The fact that $H/N$ is submetrizable, for each $N\in\mathcal{N}$, follows from \cite[Lemma~6.10.7]{AT}. Hence it suffices to verify that $\mathcal{L}$ is a weak $\sigma$-lattice for $H$. To show that 
$\mathcal{L}$ generates the topology of $H$, take an arbitrary open neighborhood $U$ of the identity $e$ in $H$. There exists a sequence $\{U_n: n\in\omega\}$ of open symmetric neighborhoods of $e$ in $H$ such that $U_0\subset U$ and $U_{n+1}^3\subset U_n$ for each $n\in\omega$. Then $N=\bigcap_{n\in\omega} U_n$ is an admissible subgroup of $H$ contained in $U$. The set $V=\pi_N(U_1)$ is open in $H/N$ and satisfies $e\in \pi_N^{-1}(V)=U_1N\subset U_1U_1\subset U_0\subset U$. Since $H$ is homogeneous, we see that the family $\mathcal{L}$ generates the topology of $H$. 

It is clear that every finite (even countable) subfamily of $\mathcal{L}$ has a lower bound in $(\mathcal{L},\prec)$ since the intersection of every countable subfamily of $\mathcal{N}$ is in $\mathcal{N}$. Finally, let $p_0\succ p_1\succ p_2\succ \cdots$ be a sequence in $\mathcal{L}$. For every $k\in\omega$, choose $N_k\in\mathcal{N}$ such that $p_k=\pi_{N_k}$ and let $N=\bigcap_{k\in\omega} N_k$. Denote by $q$ the diagonal product of the mappings $\pi_{N_k}$ with $k\in\omega$.
Since $N\subset N_k$, there exists a continuous mapping $\varphi_k\colon H/N\to H/N_k$ satisfying
$\pi_{N_k}=\varphi_k\circ\pi_N$, where $k\in\omega$. Let $\varphi$ be the diagonal product of the	family $\{\varphi_k: k\in\omega\}$. It is clear that $\varphi$ is a continuous one-to-one mapping of $H/N$ onto $q(H)$ which satisfies the equality $q=\varphi\circ\pi_N$. This proves that $\mathcal{L}$ is a weak $\sigma$-lattice for $H$.
\end{proof}

The following fact is close to \cite[Lemma~3.4]{BCDT}.

\begin{lemma}\label{Le:Adm}
Let $K$ be a closed invariant subgroup of a topological group $G$ and $\pi\colon G\to G/K$ the quotient homomorphism. Then $\pi(N)$ is an admissible subgroup of $G/K$, for each admissible
subgroup $N$ of $G$. 
\end{lemma}

\begin{proof}
Consider an admissible subgroup $N$ of $G$ and take a sequence $\{U_n: n\in\omega\}$ of symmetric open neighborhoods of the identity $e$ in $G$ such that $U_{n+1}^3\subset U_n$ for each $n\in\omega$ and $N=\bigcap_{n\in\omega} U_n$. Since the homomorphism $\pi$ is open, $V_n=\pi(U_n)$
is a symmetric open neighborhood of the identity in $G/K$ and $V_{n+1}^3\subset V_n$ for each $n\in\omega$. Clearly $H=\bigcap_{n\in\omega} V_n$ is an admissible subgroup of $G/K$. 

It remains to verify that $\pi(N)=H$. Indeed, take an arbitrary element $y\in H$. It follows from 
the definition of $H$ that $\pi^{-1}(y)\cap U_n\neq\emptyset$ for each $n\in\omega$. Note that 
$\overline{U_{n+1}}\subset U_n$ since the set $U_{n+1}$ is symmetric and $U_{n+1}^3\subset U_n$.
Hence we see that $\pi^{-1}(y)\cap \overline{U_{n+1}}$ is not empty for each $n\in\omega$. Since $p^{-1}(y)\cong K$ is compact and $\overline{U_{n+1}}\subset U_n\subset \overline{U_n}$ for each 
$n\in\omega$, we conclude that 
$$
\emptyset\neq p^{-1}(y)\cap\bigcap_{n\in\omega} \overline{U_n} = p^{-1}(y)\cap\bigcap_{n\in\omega} U_n =  p^{-1}(y)\cap N.
$$
Hence $y\in \pi(N)=H$, i.e.~$\pi(N)=H$, as claimed.
\end{proof}

\begin{lemma}\label{Le:2}
Let $X$ be a Tychonoff space such that the closure of every $G_{\delta,\Sigma}$-set in $X$ 
is a zero-set. Then the Hewitt-Hachbin completion of $X$ has the same property.
\end{lemma}

\begin{proof}
It is well known that $X$ is \emph{$G_\delta$-dense} in $\upsilon{X}$, that is, $X$ meets every non-empty $G_\delta$-set in $\upsilon{X}$. Let $P$ be a non-empty $G_{\delta,\Sigma}$-set in $\upsilon{X}$. Then $Q=P\cap X$ is a non-empty $G_{\delta,\Sigma}$-set in $X$. It follows from our assumptions on $X$ that there exists a continuous real-valued function $f$ on $X$ such that $\overline{Q}=f^{-1}(0)$, where the closure of $Q$ is taken in $X$. Denote by $g$ a continuous extension of $f$ over $\upsilon{X}$. Again, making use of the fact that $X$ is $G_\delta$-dense in $\upsilon{X}$, we deduce the equality $\overline{P}=g^{-1}(0)$, where the closure of $P$ is taken in $\upsilon{X}$. Hence $\overline{P}$ is a zero-set, as claimed.
\end{proof}

The next result follows from \cite[Corollary~5.7]{Tk94} in the case of a Tychonoff space $Y$; we extend the corresponding fact to Hausdorff spaces $Y$.

\begin{lemma}\label{Le:3}
Let a Hausdorff space $Y$ with a $G_\delta$-diagonal be a continuous image of a 
product of Lindel\"of $\Sigma$-spaces. Then $Y$ has a countable network. 
\end{lemma}

\begin{proof}
Denote by $f$ a continuous mapping of a product $X=\prod_{i\in I} X_i$ of Lindel\"of 
$\Sigma$-spaces onto $Y$. Note that for every countable set $J\subset I$, the 
subproduct $X_J=\prod_{i\in J} X_i$ is a Lindel\"of $\Sigma$-space and, hence, is Lindel\"of. 
Since $Y$ has a $G_\delta$-diagonal, it follows from \cite[Theorem~1]{Eng} that $f$ depends 
at most on countably many coordinates, so we can find  a countable set $J\subset I$ and a mapping $h_J\colon X_J\to Y$, such that $f = h_J\circ p_J$, where $p_J\colon X\to X_J$ is the projection. Since the set $J$ is countable, $X_J$ is a Lindel\"of $\Sigma$-space. It follows that $Y$ is a continuous image of a Lindel\"of $\Sigma$-space, so $Y$ is in the class $\mathcal{L}\Sigma$ (see Definition~1 and Proposition~2 in \cite{Tk15}). Finally, according to \cite[Lemma~5]{Tk15}, every space in $\mathcal{L}\Sigma$ with a $G_\delta$-diagonal has a countable network. 
\end{proof}

\begin{lemma}\label{Le:4}
Let $Y$ be an $\omega$-cellular space with a weak $\sigma$-lattice of open mappings onto regular spaces with a countable network. Then the closure in $Y$ of every $G_{\delta,\Sigma}$-set is a zero-set, so $Y$ is an Efimov space.  
\end{lemma}

\begin{proof}
Every regular space with a countable network is Lindel\"of and, hence, normal. Since $Y$ has a weak $\sigma$-lattice of continuous mappings onto normal spaces, it is Tychonoff. Let $\mathcal{R}$ be a corresponding weak $\sigma$-lattice for $Y$ of open mappings onto regular spaces with countable networks. 

Let us call a subset $F$ of $Y$ \emph{$\mathcal{R}$-cylindric} if one can find $f\in\mathcal{R}$ and
a closed subset $C$ of $p(Y)$ such that $F=p^{-1}(C)$. Clearly $\mathcal{R}$-cylindric subsets of $Y$ are closed zero-sets in $Y$. We claim that every non-empty $G_\delta$-set $P$ in $Y$ is the union of a family of $\mathcal{R}$-cylindric sets. Indeed, take an arbitrary point $x\in P$. Let $P=\bigcap_{n\in\omega} U_n$, where each $U_n$ is open in $Y$. Since $Y$ is Tychonoff and $\mathcal{R}$ generates the topology of $Y$, we can find, for every $n\in\omega$, an element $f_n$ of $\mathcal{R}$ and an open set $V_n$ in $f_n(Y)$ such that $f_n(x)\in O_n$ and $f_n^{-1}(\overline{O_n})\subset U_n$. As $\mathcal{R}$ is a weak $\sigma$-lattice for $Y$, there exists $g\in\mathcal{R}$ such that $g\prec f_n$ for each $n\in\omega$. Hence, for every $n\in\omega$, there exists a continuous mapping $h_n\colon g(Y)\to f_n(Y)$ satisfying $f_n=h_n\circ g$. Let $W_n=h_n^{-1}(O_n)$, where $n\in\omega$. Then $g(x)\in W_n$ and $g^{-1}(\overline{W_n})\subset g^{-1}(h_n^{-1}(\overline{O_n})) = f_n^{-1}(\overline{O_n})\subset  U_n$. Put $C=\bigcap_{n\in\omega} \overline{W_n}$. Then $C$ is a closed subset of $g(Y)$ containing the point $g(x)$. Since the space $g(Y)$ is regular and has a countable network, $C$ is a zero-set in $g(Y)$. Hence $g^{-1}(C)$ is a zero-set in $Y$. It also follows from our definition of the sets $W_n$ that
$$ 
x\in g^{-1}(C)=\bigcap_{n\in\omega} g^{-1}(\overline{W_n}) \subset 
\bigcap_{n\in\omega} U_n\subset F.
$$
This implies our claim.

Let $F$ be a $G_{\delta,\Sigma}$-set in $Y$. We have just proved that there exists a family 
$\gamma$ of $\mathcal{R}$-cylindric sets in $Y$ such that $F=\bigcup\gamma$. Since $Y$ 
is $\omega$-cellular, $\gamma$ contains a countable subfamily $\lambda$ such that 
$\bigcup\lambda$ is dense in $\bigcup\gamma$. Let $\lambda=\{F_n: n\in\omega\}$. For 
every $n\in\omega$, take an element $g_n\in\mathcal{R}$ and a closed subset $C_n$ of 
$g_n(Y)$ such that $F_n=g_n^{-1}(C_n)$. Since $\mathcal{R}$ is a weak $\sigma$-lattice 
for $Y$, there exists $g\in\mathcal{R}$ satisfying $g\prec g_n$ for each $n\in\omega$. It follows 
from our choice of the mapping $g$ that $\bigcup\lambda=g^{-1}(g(\bigcup\lambda))$. Since the mapping $g$ is open and $\bigcup\lambda$ is dense in $\bigcup\gamma$, we have that 
$$
\overline{\bigcup\lambda} = g^{-1}g \Big(\hskip1pt\overline{\bigcup\lambda}\hskip1pt\Big) =
g^{-1}\left(g \big(\hskip1pt\overline{\bigcup\gamma}\big)\right)  \supseteq
g^{-1}g\big(\bigcup\gamma\big)\supseteq \bigcup\lambda,
$$
that is, $\bigcup\lambda$ is dense in $\bigcup\gamma$. This implies the equalities $\overline{\bigcup\lambda} = g^{-1}g \Big(\hskip1pt\overline{\bigcup\lambda}\hskip1pt\Big) =\overline{\bigcup\gamma}$. Since $g\big(\overline{\bigcup\lambda}\big)=\overline{g(\bigcup\lambda)}$ is a closed subset of $g(Y)$
(once again we apply the fact that the mapping $g$ is open), the sets $C=\overline{g(\bigcup\lambda)}$ and the inverse image of $C$ under $g$ are zero-sets. This completes the proof of the lemma.
\end{proof}

Now we are ready to prove the next main result of this section.

\begin{theorem}\label{Th:1y}
Let a topological group $H$ be a continuous homomorphic image of an almost connected pro-Lie group $G$. Then the closure in $H$ of every $G_{\delta,\Sigma}$-set is a zero-set, so $H$ is an Efimov space. The Hewitt--Nachbin completion of $H$ is also an Efimov space.
\end{theorem}

\begin{proof}
We recall that the group $G$ is \emph{homeomorphic} to the product $\R^\kappa\times K$, 
where $\R$ is the real line with the usual topology, $\kappa$ is a cardinal, and $K$ is a 
compact topological group. Let $\mathcal{N}$ be the family of all admissible subgroups of $H$. 
Then, by Lemma~\ref{Le:1}, the family $\{\pi_N: N\in\mathcal{N}\}$ is a weak $\sigma$-lattice for 
$H$ which consists of open mappings onto Hausdorff spaces with a $G_\delta$-diagonal. Since 
$G$ is homeomorphic to the product of Lindel\"of $\Sigma$-spaces and $H$ is a continuous homomorphic image of $G$, it follows from Lemma~\ref{Le:3} that the quotient space $H/N$ 
has a countable network, for each $N\in\mathcal{N}$. We know in addition that the space $H$ 
is $\omega$-cellular. Hence Lemma~\ref{Le:4} implies that the closure of every 
$G_{\delta,\Sigma}$-set in $H$ is a zero-set. Therefore $H$ is an Efimov space. 

Finally, we apply Lemma~\ref{Le:2} to conclude that the closure in $\upsilon{H}$ of 
every $G_{\delta,\Sigma}$-set is a zero-set, so $\upsilon{H}$ is also an Efimov space. 
\end{proof}

\begin{corollary}\label{Cor:AlCy}
Every almost connected pro-Lie group $H$ is an Efimov space. 
\end{corollary}

%%%%%%%%%%%%%%%%%%%%%%%%%%%%%%%%%%%%%%%
\section{Strong $\sigma$-lattices of homomorphisms}\label{Strong_lattices}
%%%%%%%%%%%%%%%%%%%%%%%%%%%%%%%%%%%%%%%
The following stronger version of the notion of weak $\sigma$-lattice plays
an important role in the rest of the article. It enables us to generalize several 
results of Section~\ref{R_pro-Lie} (see Theorems~\ref{Th:2} and Corollary~\ref{Cor:EXT}).

\begin{definition}[See \cite{Tk94}]\label{Def:2}
A family $\mathcal{L}$ of continuous mappings of a space $X$ elsewhere is said to be 
a \emph{strong $\sigma$-lattice} for $X$ if it has the following properties:
  \begin{enumerate}
    \item[(1)] $\mathcal{L}$ generates the topology of $X$;
    \item[(2)] every finite subfamily of $\mathcal{L}$ has a lower bound in $\mathcal{L}$ with 
                   respect to the partial preorder $\prec$ introduced in Definition~\ref{Def:1};
    \item[(3)] for every decreasing sequence $p_0,p_1,p_2,\ldots$ in $(\mathcal{L},\prec)$, 
           the diagonal product of the family $\{p_n: n\in\omega\}$, say, $p_\omega$ belongs to 
           $\mathcal{L}$, and if a sequence $\{x_n: n\in\omega\}\subset X$ has the property 
           $p_k(x_n)=p_k(x_k)$ whenever $k<n$, then there exists $x\in X$ such that $p_n(x)=x_n$ 
           for each $n\in\omega$. We will also denote the mapping $p_\omega\in\mathcal{L}$ by 
           $\lim_{n\in\omega} p_n$.
  \end{enumerate}
\end{definition}

It is worth mentioning that condition (3) of the above definition is equivalent to saying that the limit space of the inverse sequence $\{p_n(X),\ p^{n}_k: k<n,\ k,n\in\omega\}$ is naturally homeomorphic
to $p_\omega(X)$, where $p^n_k=p_k\circ p_n^{-1}$ if $k<n$. A typical example of a strong $\sigma$-lattice is the family of all projections of the product space $X=\prod_{i\in I} X_i$ onto countable subproducts. 

\begin{definition}\label{Def:3}
A family $\mathcal{L}$ of continuous mappings of a space $X$ elsewhere has the \emph{factorization property} if for every continuous mapping mapping $f\colon X\to Y$ to a second countable Hausdorff space $Y$, one can find an element $p\in\mathcal{L}$ with $p\prec f$.
\end{definition}
 
Let $X=\prod_{i\in I} X_i$ be a product of separable spaces. According to a theorem of Glicksberg in \cite{Gli}, the family of projections of $X$ onto countable subproducts has the factorization property. 

We recall that a continuous mapping $f\colon X\to Y$ is called \emph{$d$-open} (or \emph{nearly open}) if for every open set $U\subset X$, there exists an open set $V\subset Y$ such that $f(U)$ 
is a dense subset of $V$ (see \cite{Tk81}).
 
\begin{lemma}\label{Le:5.0}
Let $\mathcal{L}$ be a weak $\sigma$-lattice of $d$-open, quotient mappings for a weakly 
Lindel\"of space $X$. Then $\mathcal{L}$ has the factorization property.
\end{lemma} 
 
\begin{proof}
Let $f$ be a continuous real-valued function on $X$. For a non-empty set $A\subset X$, we put 
$$
osc(f,A)=\sup\{|f(x)-f(y)|: x,y\in A\}.
$$
Clearly $osc(f,A)$ is finite iff $f$ is bounded on $A$; otherwise $osc(f,A)$ is defined to be $\infty$. 
 
Since $f$ is continuous, for every $x\in X$ and every positive integer $n$, there exists an open neighborhood $U_n(x)$ of $x$ in $X$ such that $osc(f,U_n(x))<1/n$. Choose an element $p_{n,x}
\in\mathcal{L}$ and an open set $V_n(x)$ in $p_{n,x}(X)$ such that $x\in p_{n,x}^{-1}(V_n(x))\subset U_n(x)$. Let $O_n(x)=p_{n,x}^{-1}(V_n(x))$. It follows from $O_n(x)\subset U_n(x)$ that $osc(f,O_n(x))<1/n$. Since $X$ is weakly Lindel\"of we can find, for every integer $n>0$, a 
countable set $B(n)\subset X$ such that $\bigcup_{y\in B(n)} O_n(y)$ is dense in $X$. 
 
Since $\mathcal{L}$ is a weak $\sigma$-lattice for $X$ and the sets $B(n)$ are countable, there exists $p\in\mathcal{L}$ such that $p\prec p_{n,x}$ for all $n\in\N$ and $x\in B(n)$. We claim that $p\prec f$. First we show that if $x,y\in X$ and $p(x)=p(y)$, then $f(x)=f(y)$. Suppose for a contradiction that there exist $x,y\in X$ such that $p(x)=p(y)$ and $f(x)\neq f(y)$. Then $|f(x)-f(y)|\geq 3/n$ for some integer $n>0$. Since the mapping $p$ is $d$-open, there are open sets $W_x$ and $W_y$ in $p(X)$ such that $p(O_n(x))$ is dense in $W_x$ and $p(O_n(y))$ is dense in $W_y$. It is clear that $W^*=W_x\cap W_y$ is an open neighborhood of the point $p(x)$. Since the set $\bigcup_{z\in B(n)} O_n(z)$ is dense in $X$, we can find $z\in B(n)$ such that $W^*\cap p(O_n(z))\neq\emp$. It follows from $p\prec p_{n,z}$ that $O_n(z)=p^{-1}p(O_n(z))$ and $p(O_n(z))$ is open in $p(X)$. Since the sets $p(O_n(x))\cap W^*$ and $p(O_n(y))\cap W^*$ are dense in $W^*$, we see that $p(O_n(x))\cap p(O_n(z))\neq\emp$ and $p(O_n(y))\cap p(O_n(z))\neq\emp$. Therefore, $O_n(x)\cap O_n(z)\neq\emp$ and $O_n(y)\cap O_n(z)\neq\emp$. Take elements $a\in O_n(x)\cap O_n(z)$ and $b\in O_n(y)\cap O_n(z)$. Our choice of the sets $O_n(x)$, $O_n(y)$ and $O_n(z)$ implies that 
$$
|f(x)-f(y)|\leq |f(x)-f(a)| + |f(a)-f(b)| + |f(b)-f(y)| < 1/n + 1/n + 1/n = 3/n,
$$
thus contradicting the inequality $|f(x)-f(y)|\geq 3/n$. This proves our claim. 

It follows from the claim that there exists a real-valued function  $\varphi$ on $p(X)$ such that $f=\varphi\circ p$. Since $p$ is quotient and $f$ is continuous, the function $\varphi$ is continuous as well. Hence $p\prec f$ and the lattice $\mathcal{L}$ has the factorization property.
\end{proof}
 
The next result complements Lemma~\ref{Le:5.0} in the case when the mappings of 
$\mathcal{L}$ are open. 
 
\begin{lemma}\label{Le:5}
Let $\mathcal{L}$ be a weak $\sigma$-lattice of open mappings for a weakly Lindel\"of space $X$. Then, for every continuous mapping $f\colon X\to Y$ to a Hausdorff space $Y$ with a countable network, one can find $p\in\mathcal{L}$ satisfying $p\prec f$.
\end{lemma}

\begin{proof}
The family $\mathcal{L}$ has the factorization property by Lemma~\ref{Le:5.0}. Let $i\colon Y\to Z$ 
be a continuous one-to-one mapping onto a Hausdorff space $Z$ of countable weight. Then $i\circ{f}$ is a continuous mapping of $X$ onto $Z$. Since $\mathcal{L}$ has the factorization property, we can find $p\in\mathcal{L}$ and a continuous mapping $\varphi\colon p(X)\to Z$ satisfying $i\circ{f}=\varphi\circ{p}$. Let $h=i^{-1}\circ\varphi$, $h\colon p(X)\to Y$. Since the mapping $p$ is continuous and open, it follows from the equality $f=h\circ{p}$ that $h$ is continuous. This proves the lemma. 
\end{proof}

\begin{lemma}\label{Le:wL}
Let a space $X$ have a strong $\sigma$-lattice of open mappings onto weakly Lindel\"of spaces. 
Then $X$ is also weakly Lindel\"of.
\end{lemma}

\begin{proof}
Denote by $\mathcal{L}$ a strong $\sigma$-lattice for $X$ consisting of open mappings onto 
weakly Lindel\"of spaces. Let $\gamma$ be an open cover of $X$. Since $\mathcal{L}$ generates 
the topology of $X$, we can assume without loss of generality that every element $U\in\gamma$ 
has the form $U=p^{-1}(V)$ for some $p\in\mathcal{L}$ and an open set $V\subset p(X)$. 

Take an arbitrary element $p_0\in\mathcal{L}$. Since $p_0$ is open and the image $X_0=p_0(X)$ 
is weakly Lindel\"of, we can find a countable subfamily $\lambda_0$ of $\gamma$ such that $p_0(\bigcup\lambda_0)$ is dense in $X_0$. Assume that for some $n\in\omega$ we have defined elements $p_0,\ldots,p_n\in\mathcal{L}$ and countable subfamilies $\lambda_0,\ldots,\lambda_n$ 
of $\gamma$ satisfying the following conditions:
\begin{enumerate}
\item[(i)]   $p_n\prec\cdots\prec p_0$;
\item[(ii)]  $\lambda_0\subset \cdots\subset\lambda_n$; 
\item[(iii)] $p_k(\bigcup\lambda_k)$ is dense in $X_k$ for each $k\leq n$;
\item[(iv)] if $k<n$, then $U=p_n^{-1}p_n(U)$ for each $U\in \lambda_k$.
\end{enumerate}
Since $\lambda_n$ is countable, there exists an element $p_{n+1}\in\mathcal{L}$ such that 
$p_{n+1}\prec p_n$ and $U=p_{n+1}^{-1} p_{n+1}(U)$ for each $U\in\lambda_n$. Again, the space $X_{n+1}=p_{n+1}(X)$ is weakly Lindel\"of, so there exists a countable subfamily $\lambda_{n+1}$ 
of $\gamma$ such that $\lambda_n\subset\lambda_{n+1}$ and $p_{n+1}(\bigcup\lambda_{n+1})$ 
is dense in $X_{n+1}$. This finishes our construction of the sequences $\{p_n: n\in\omega\}\subset\mathcal{L}$ and $\{\lambda_n: n\in\omega\}$. 

Since $\mathcal{L}$ is a strong $\sigma$-lattice for $X$, the diagonal product of the family $\{p_n: n\in\omega\}$, say, $p$ is in $\mathcal{L}$. Clearly $\lambda=\bigcup_{n\in\omega}\gamma_n$ is a countable subfamily of $\gamma$. We claim that $\bigcup\lambda$ is dense in $X$. Indeed, it follows from (iv) that $U=p^{-1}p(U)$ for each $U\in\lambda$. Hence $\bigcup\lambda=p^{-1}p(\bigcup\lambda)$. Further, conditions (i)--(iii) imply that $p(\bigcup\lambda)$ is dense in $p(X)$. Since the mapping $p$ is open, we conclude that $\bigcup\lambda$ is dense in $X$. Therefore the space $X$ 
is weakly Lindel\"of.
\end{proof}

Assume that $f\colon X\to Y$ is a homeomorphism and $\mathcal{L}_X$ and $\mathcal{L}_Y$ are (weak or strong) $\sigma$-lattices of mappings for the spaces $X$ and $Y$, respectively. We say that 
$\mathcal{L}_X$ and $\mathcal{L}_Y$ are \emph{isomorphic} if one can find a bijection $\Phi\colon\mathcal{L}_X\to\mathcal{L}_Y$ and a family of mappings $\{f_p: p\in\mathcal{L}_X\}$, where $f_p\colon p(X)\to \Phi(p)(Y)$ is a homeomorphism such that the following diagram commutes, for each 
$p\in\mathcal{L}_X$.
\[
\xymatrix{X\ar@{>}[r]^{f} \ar@{>}[d]_{p} & Y
\ar@{>}[d]^{\Phi(p)}\\
p(X)\ar@{>}[r]^{f_p}  &  \Phi(p)(Y) }
\]
If this happens, we say that the homeomorphism $f$ is \emph{induced} by an isomorphism of the lattices $\mathcal{L}_X$ and $\mathcal{L}_Y$.

The following lemma describes a situation when a homeomorphism between two spaces with 
\lq{good\rq} $\sigma$-lattices of open mappings is generated by an isomorphism of cofinal 
$\sigma$-sublattices of the two lattices. 

\begin{lemma}\label{Le:6}
Let $f$ be a homeomorphism of a space $X$ onto $Y$. Let also $\mathcal{L}_X$ and
$\mathcal{L}_Y$ be a strong $\sigma$-lattice for $X$ and a weak $\sigma$-lattice for $Y$, respectively. If the lattices $\mathcal{L}_X$ and $\mathcal{L}_Y$ consist of open mappings onto spaces with a countable network, then the homeomorphism $f$ is induced by an isomorphism of cofinal strong $\sigma$-sublattices of $\mathcal{L}_X$ and $\mathcal{L}_Y$. In particular, 
$\mathcal{L}_Y$ contains a cofinal strong $\sigma$-lattice for $Y$ with the factorization property. 
\end{lemma}

\begin{proof}
Clearly every space with a countable network is separable and, hence, weakly Lindel\"of. By 
Lemma~\ref{Le:wL}, the space $X$ is weakly Lindel\"of, and so is $Y$. Therefore Lemma~\ref{Le:5.0} implies that the lattices $\mathcal{L}_X$ and $\mathcal{L}_Y$ have the factorization property.

Let $p\in\mathcal{L}_X$ and $q\in\mathcal{L}_Y$. We call the pair $(p,q)$ \emph{coherent} if there 
exists a homeomorphism $h\colon p(X)\to q(Y)$ satisfying $q\circ{f}=h\circ{p}$. Assume that 
$(p_1,q_1)$ is another coherent pair such that $p_1\prec p$ and $q_1\prec q$, and let a homeomorphism $h_1\colon p_1(X)\to q_1(Y)$ satisfy $q_1\circ{f_1}=h_1\circ{p_1}$. According 
to our assumptions there exist continuous mappings $\varphi\colon p_1(X)\to p(X)$ and $\psi\colon q_1(Y)\to q(Y)$ such that $p=\varphi\circ p_1$ and $q=\psi\circ q_1$. One can easily verify that the mappings $h,h_1$ and $\varphi,\psi$ satisfy $h\circ\varphi=\psi\circ h_1$. This simple observation will be used in the sequel without mentioning.  

Consider a sequence $\{(p_n,q_n): n\in\omega\}$, where $p_n\in\mathcal{L}_X$, $q_n\in 
\mathcal{L}_Y$, the pair $(p_n,q_n)$ is coherent, and $p_{n+1}\prec p_n$, $q_{n+1}\prec q_n$
for each $n\in\omega$. Since $\mathcal{L}_X$ is a strong $\sigma$-lattice for $X$, the mapping $p_\omega=\lim_{n\in\omega} p_n$ is in $\mathcal{L}_X$. We can identify $p_\omega$ with the diagonal product of the family $\{p_n: n\in\omega\}$. Similarly, since $\mathcal{L}_Y$ is a weak $\sigma$-lattice for $Y$, there exists $q=w$-$\lim_{n\in\omega} q_n$ in $\mathcal{L}_Y$. Let us show that the pair 
$(p_\omega,q)$ is coherent. In the sequel we put $X_n=p_n(X)$ and $Y_n=q_n(Y)$, where 
$n\in\omega$. 

First, for every $n\in\omega$, take a homeomorphism $f_n\colon X_n\to Y_n$ witnessing the coherence of the pair $(p_n,q_n)$. Let $q_\omega$ be the diagonal product of the family $\{q_n: n\in\omega\}$. Then there exists a continuous one-to-one mapping $i\colon q(Y)\to q_\omega(Y)$ satisfying $i\circ{q}=q_\omega$. It will be shown below that $i$ is a homeomorphism of $q(Y)$ onto 
$Y_\omega$. 

For all $n,k\in\omega$ with $k<n$, let also $q^n_k\colon Y_n\to Y_k$ be a continuous mapping satisfying $q_k=q^n_k\circ{q_n}$. Denote by $Y_\omega$ the inverse limit of the sequence $\{Y_n,\ q^{n+1}_n: n\in\omega\}$. Then $q_\omega(Y)$ is a dense subspace of $Y_\omega$. For every 
$n\in\omega$, there exist continuous mappings $q^\omega_n\colon Y_\omega\to Y_n$ and $q^{n+1}_n\colon Y_{n+1}\to Y_n$ satisfying $q_n=q^\omega_n\circ q_\omega$ and $q^{n+1}_n\circ q_{n+1}=
q_n$. Similarly, there exist continuous mappings $p^\omega_n\colon X_\omega\to X_n$ and $p^{n+1}_n\colon X_{n+1}\to X_n$ satisfying $p_n=p^\omega_n\circ p_\omega$ and $p^{n+1}_n\circ p_{n+1}=
p_n$, where $X_\omega=p_\omega(X)$.

Let $f_\omega\colon X_\omega\to Y_\omega$ be a continuous mapping satisfying $q_n\circ f=f_n\circ p_n$ for each $n\in\omega$, i.e.~$f_\omega$ is the \emph{limit} of the sequence 
$\{f_n: n\in\omega\}$. It is easy to see that the equality $q_\omega\circ{f}=f_\omega\circ p_\omega$ holds. 
\[
\xymatrix{X\ar@{>}[r]^{f} \ar@{>}[d]_{p_\omega} & Y\ar@/^2.5pc/[dd]^{q_\omega}
\ar@{>}[d]^{q}\\
X_\omega\ar@{-->}[r]^{g} \ar@{>}[dr]_{f_\omega} &  q(Y) \ar@{>}[d]^{i}\\
 & Y_\omega}
\]
We claim that $q_\omega(Y)=Y_\omega$ or, equivalently, $i(q(Y))=Y_\omega$. Indeed, take an arbitrary point $y\in Y_\omega$. For every $n\in\omega$, let $y_n=q^\omega_n(y)$ and choose a 
point $x_n\in X_n$ with $f_n(x_n)=y_n$. It is easy to verify that $p^{n+1}_n(x_{n+1})=x_n$ for 
each $n\in\omega$. To see this, we note that 
$$
f_n(p^{n+1}_n(x_{n+1}))=q^{n+1}_n(f_{n+1}(x_{n+1}))=q^{n+1}_n(y_{n+1})=y_n,
$$
so we have the equality $f_n(p^{n+1}_n(x_{n+1}))=y_n=f_n(x_n)$. Since $f_n$ is a bijection of $X_n$ onto $Y_n$, we conclude that $p^{n+1}_n(x_{n+1})=x_n$, as claimed. As $\mathcal{L}_X$ is a strong 
$\sigma$-lattice for $X$, there exists $x\in X$ such that $p_n(x)=x_n$ for each $n\in\omega$. 
Let $z=f(x)$. Then $q_n(z)=q_n(f(x)) = f_n(p_n(x))=f_n(x_n)=y_n$ for each $n\in\omega$, whence it follows that $q_\omega(z)=y$. We have thus proved that $q_\omega$ maps $Y$ onto $Y_\omega$, 
so $i(q(Y))=Y_\omega$. This implies that $i$ is a continuous bijection. The equality $f_\omega\circ 
p_\omega=q_\omega\circ f$ also implies that $f_\omega$ maps $X_\omega$ onto $Y_\omega$.

Notice that $f_\omega$ is a homeomorphism between $X_\omega$ and $Y_\omega$ as the limit of the homeomorphisms $f_n$, with $n\in\omega$. Since $f,p_\omega,q$ are continuous open surjective mappings, so is $g=i^{-1}\circ f_\omega$. The equality $f_\omega=i\circ{g}$ and the fact that $i$ and 
$f_\omega$ are continuous bijections together imply that $g$ is also a bijection. Hence $g$ and $i$ are homeomorphisms. The equality $q\circ f=g\circ p_\omega$ proves that the pair $(p_\omega,q)$ is coherent, as claimed. In other words, the lattices $\mathcal{L}_X$ and $\mathcal{L}_Y$ are \lq{closed\rq} with respect to taking limits of sequences of coherent pairs of mappings. 

Finally we show that 
$$
\mathcal{R}=\{p\times q: p\in \mathcal{L}_X,\ q\in\mathcal{L}_Y,\ (p,q) \mbox{ is coherent}\}
$$
is a cofinal $\sigma$-sublattice of the product lattice 
$$
\mathcal{L}_X\times\mathcal{L}_Y=\{p\times q: p\in \mathcal{L}_X,\ q\in\mathcal{L}_Y\}.
$$ 
To this end, take arbitrary elements $p_0\in\mathcal{L}_X$ and $q_0\in\mathcal{L}_Y$. Our aim is to find $p_\omega\in\mathcal{L}_X$, $q_\omega\in\mathcal{L}_Y$, and a homeomorphism $f_\omega\colon p_\omega(X)\to q_\omega(Y)$ satisfying $p_\omega\prec p_0$, $q_\omega\prec q_0$, and $f_\omega\circ p_\omega=q_\omega\circ{f}$. Since the family $\mathcal{L}_X$ has the factorization property and $q_0\circ{f}$ is a continuous mapping of $X$ to the space $q_0(Y)$ with a countable network, we apply Lemma~\ref{Le:5} to find $p_1\in\mathcal{L}_X$ and a continuous mapping $\varphi_1\colon p_1(X)\to q_0(Y)$ such that $p_1\prec p_0$ and $q_0\circ{f}=\varphi_1\circ p_1$. Similarly, using the factorization property of $\mathcal{L}_Y$ and Lemma~\ref{Le:5}, we find $q_1\in\mathcal{L}_Y$ and a continuous mapping $\psi_1\colon q_1(Y)\to p_1(X)$ such that $q_1\prec q_0$ and $\psi_1\circ q_1\circ f=p_1$. Continuing this construction we define sequences $\{p_n: n\in\omega\}\subset\mathcal{L}_X$ and $\{q_n: n\in\omega\}\subset\mathcal{L}_Y$ such that 
$$
p_{n+1}\prec p_n,\ q_{n+1}\prec q_n,\, \mbox{ and }\,  p_{n+1}\prec  q_n\circ f\prec p_n
$$ 
for each $n\in\omega$. Let $\varphi_{n+1}\colon p_{n+1}(X)\to q_n(Y)$ and $\psi_n\colon q_n(Y)\to p_n(X)$ be continuous mappings such that $\psi_n\circ q_n\circ f = p_n$ and $q_n\circ f = \varphi_{n+1}\circ p_{n+1}$. Since $\mathcal{L}_X$ is a strong $\sigma$-lattice for $X$, there exists $p_\omega=\lim_{n\in\omega} p_n$ in $\mathcal{L}_X$ and, similarly, one can find $q=w$-$\lim_{n\in\omega} q_n$ in $\mathcal{L}_Y$. It remains to verify that the pair $(p_\omega,q)$ is coherent. 

For every $n\in\omega$, $p^\omega_n=p_n\circ p_\omega^{-1}$ is a continuous open mapping of $p_\omega(X)$ onto $p_n(X)$. Denote by $Y_\omega$ the limit space of the inverse sequence $\{q_n(Y),\ q^{n+1}_n: n\in\omega\}$, where $q^{n+1}_n=q_{n+1}^{-1}\circ q_n$ is a continuous open mapping of $q_{n+1}(Y)$ onto $q_n(Y)$. Let $q_\omega\colon Y\to Y_\omega$ be the diagonal product of the family $\{q_n: n\in\omega\}$. There exists a continuous one-to-one mapping $i\colon q(Y)\to Y_\omega$ satisfying $q_\omega=i\circ{q}$. Since $q_n$ is an onto mapping for each $n\in\omega$, the image $q_\omega(Y)$ is dense in $Y_\omega$. For every $n\in\omega$, let $q^\omega_n\colon Y_\omega\to Y_n$ be a continuous mapping satisfying $q_n=q^\omega_n\circ q_\omega$. 

Let $\varphi_\omega\colon p_\omega(X)\to Y_\omega$ be the limit of the mappings $\{\varphi_n: n\in\omega,\ n\geq 1\}$. Then $\varphi_\omega$ is continuous and satisfies the equality $\varphi_\omega\circ p_\omega=q_\omega\circ{f}$, so $\varphi_\omega(p_\omega(X))=q_\omega(Y)$ is a dense subspace of $Y_\omega$. Let also $\psi_\omega\colon Y_\omega\to p_\omega(X)$ be the limit of the mappings $\psi_n$ with $n\in\omega$. Clearly the equalities $q^\omega_n\circ\varphi_\omega = \varphi_{n+1}\circ p^\omega_{n+1}$ and $p^\omega_n\circ\psi_\omega = \psi_n\circ q^\omega_n$ hold.
Therefore the following diagram commutes.
\[
\xymatrix{X\ar@{>}[r]^{f} \ar@{>}[dd]_{p_\omega} & Y \ar@/^2.5pc/[dd]^{q_\omega} \ar@{>}[d]^{q}\\
 &  q(Y) \ar@{>}[d]^{i}\\
p_\omega(X) \ar@/^0.2pc/[r]^{\,\,\varphi_\omega} \ar@{-->}[ur]^{g} \ar@{>}[d]_{p^\omega_{n+1}} 
\ar@/_{2.8pc}/[dd]_{p^\omega_n}   
& Y_\omega \ar@{>}[d]^{q^\omega_{n+1}} \ar@/^2.5pc/[dd]^{q^\omega_n} 
\ar@/^0.2pc/[l]^{\,\,\psi_\omega} \\
p_{n+1}(X) \ar@{>}[dr]^{\varphi_{n+1}}  \ar@{>}[d]_{p^{n+1}_n} &  q_{n+1}(Y) \ar@{>}[l]_{\psi_{n+1}}  \ar@{>}[d]^{q^{n+1}_n}\\
p_n(X)   &   q_n(Y)  \ar@{>}[l]_{\psi_n}   }  
\]

Let us verify that $\varphi_\omega$ and $\psi_\omega$ are homeomorphisms. Making use of the commutativity of the above diagram, we obtain that
$$
p^\omega_n\circ \psi_\omega\circ\varphi_\omega = \psi_n\circ q^\omega_n\circ\varphi_\omega = p^\omega_n
$$
and, similarly,
$$
q^\omega_n\circ \varphi_\omega\circ\psi_\omega = \varphi_{n+1}\circ p^\omega_{n+1}\circ \psi_\omega = q^\omega_n
$$
for each $n\in\omega$. Since the families $\{p^\omega_n: n\in\omega\}$ and $\{q^\omega_n: n\in\omega\}$ generate the topologies of $p_\omega(X)$ and $Y_\omega$, respectively, we conclude that 
$\psi_\omega\circ\varphi_\omega$ and $\varphi_\omega\circ\psi_\omega$ are identity mappings. Hence both $\varphi_\omega$ and $\psi_\omega$ are surjective homeomorphisms. It also follows from $q_\omega(Y)=\varphi_\omega(p_\omega(X))$ that the mapping $q_\omega$ is surjective. Since $q_\omega=i\circ{q}$, the one-to-one mapping $i$ is surjective and, hence, bijective. Let $g=i^{-1}\circ\varphi_\omega$. As in the first part of the proof, we deduce form the equality $q\circ f=g\circ p_\omega$ that $g$ is continuous. Since $\varphi_\omega=i\circ{g}$ is a homeomorphism of $p_\omega(X)$ onto $Y_\omega$, it is clear that $i$ and $g$ are homeomorphisms. Therefore the mappings $p_\omega$ and $q$ are coherent. 

Since $p_\omega\prec p_0$ and $q\prec q_\omega\prec q_0$, we infer that the family $\mathcal{R}$ is cofinal in $\mathcal{L}_X\times\mathcal{L}_Y$. The first part of our proof shows that $\mathcal{R}$ is $\sigma$-closed in $\mathcal{L}_X\times\mathcal{L}_Y$. Hence the second \lq{coordinates\rq} of the coherent pairs $(p,q)$ form a cofinal $\sigma$-sublattice of $\mathcal{L}_Y$. This sublattice has the factorization property as a cofinal sublattice of $\mathcal{L}_Y$. 
\end{proof}

\begin{lemma}\label{Le:7}
Let a topological group $G$ be homeomorphic as a space to a product $X=\prod_{i\in I} X_i$ of spaces with a countable network. Then $G$ has a strong $\sigma$-lattice of continuous open homomorphisms onto topological groups with a countable network.
\end{lemma}

\begin{proof}
Let $f\colon X\to G$ be a homeomorphism. For every non-empty set $J\subset I$, let 
$p_J\colon X\to X_J=\prod_{i\in J} X_i$ be the projection. Then 
$$ 
\mathcal{L}_X=\{p_J: J\subset I,\ |J|\leq\omega\}
$$
is a strong $\sigma$-lattice of open mappings for the space $X$. Notice that the space   
$X_J=p_J(X)$ has a countable network, for each countable set $J\subset I$.

Denote by $\mathcal{N}$ the family of admissible subgroups of the group $G$. According to 
Lemma~\ref{Le:1}, the family
$$
\mathcal{L}_G=\{\pi_N: N\in\mathcal{N}\}
$$
is a weak $\sigma$-lattice for $G$, where $\pi_N\colon G\to G/N$ is a quotient mapping onto the 
left coset space $G/N$. By Lemma~\ref{Le:1}, the space $G/N$ is submetrizable and, hence, has 
a $G_\delta$-diagonal. Applying Lemma~\ref{Le:3} we see that the space $G/N$ has a countable network, for each $N\in\mathcal{N}$. Therefore the required conclusion follows from 
Lemma~\ref{Le:6}.
\end{proof}

\begin{problem}\label{Prob:71}
Let a topological group $G$ be homeomorphic as a space to the product $H=\prod_{i\in I} H_i$ 
of Lindel\"of $\Sigma$-groups. Does $G$ have a strong $\sigma$-lattice of continuous open homomorphisms onto Lindel\"of $\Sigma$-groups?
\end{problem}

Let us recall that a space $X$ is said to be \emph{pseudo-$\omega_1$-compact} if every locally
finite family of open sets in $X$ is countable. It is clear that every weakly Lindel\"of space is 
pseudo-$\omega_1$-compact, while simple examples show that the converse is false. 

\begin{lemma}\label{Le:wL2}
Let $K$ be a compact subgroup of a topological group $G$. If the quotient space $G/K$ is 
weakly Lindel\"of (or pseudo-$\omega_1$-compact), so is $G$.
\end{lemma}

\begin{proof}
Assume that $G/K$ is weakly Lindel\"of and let $\gamma$ be an open cover of $G$. Denote by $\pi$ the quotient mapping of $G$ onto the left coset space $G/K$. Since each fiber of $\pi$ is compact, for every $y\in G/K$ there exists a finite subfamily $\mu_y\subset\gamma$ such that $\pi^{-1}(y)\subset\bigcup\mu_y$. The mapping $\pi$ is closed, so for every $y\in G/K$ we can find an open neighborhood $U_y$ of $y$ in $G/K$ such that $\pi^{-1}(U_y)\subset\bigcup\mu_y$. Since the space $G/K$ is weakly Lindel\"of, the open cover $\{U_y: y\in G/K \}$ of $G/K$ contains a countable subfamily whose union is dense in $G/K$. Let $\{U_y: y\in C\}$ be such a subfamily, where the set 
$C\subset G/K$ is countable. We claim that the union of the countable family $\mu=\bigcup_{y\in C} \mu_y\subset\gamma$ is dense in $G$. 

Indeed, it follows from our choice of the set $C$ that $O=\bigcup_{y\in C} U_y$ is dense in $G/K$. 
Since $\pi$ is a continuous open mapping, the set $\pi^{-1}(O)$ is dense in $G$. It remains to note
that $\pi^{-1}(O)\subset \bigcup\mu$, so the set $\bigcup\mu$ is dense in $G$. Since $\mu$ is a
countable subfamily of $\gamma$, we conclude that $G$ is weakly Lindel\"of.\smallskip

Finally, let the space $G/K$ be pseudo-$\omega_1$-compact and suppose for a contradiction that 
$\gamma$ is an uncountable locally finite family of open sets in $G$. Since $K$ is compact, each coset $xK$ in $G$ meets at most finitely many elements of $\gamma$. Hence the family $\lambda=
\{\pi(U): U\in\gamma\}$ is also uncountable. This fact and the pseudo-$\omega_1$-compactness of $G/K$ together imply that the family $\lambda$ accumulates at some point $y\in G/K$. Since the quotient mapping $\pi\colon G\to G/K$ is perfect, the family $\gamma$ must accumulate at a 
point $x\in G$ with $\pi(x)=y$, which is a contradiction. So the space $G$ has to be 
pseudo-$\omega_1$-compact.
\end{proof}

The following lemma is a special case of \cite[Proposition~3.4]{Ar80}. It will be used in the proof of
Theorem~\ref{Th:Ext}.

\begin{lemma}\label{Le:CM}
If $C$ is a compact subspace of a topological group with countable pseudocharacter, then
$C$ has a countable base.
\end{lemma}

We also need another auxiliary fact established in the proof of \cite[Theorem~1.5.20]{AT}.
It explains, informally speaking, to what extent one can restore the topology of a given group
$G$ in terms of a subgroup $K$ of $G$ and the quotient space $G/K$.

\begin{lemma}\label{Le:AT}
Let $K$ be a closed subgroup of a topological group $G$ with identity $e$ and $\pi\colon G\to G/K$
be the quotient mapping onto the left coset space $G/K$. Assume that symmetric open neighborhoods $O,\,V,\,W,\,W',\,U$ of $e$ in $G$ satisfy $V^2\subset O$, $W\cap K\subset V$, $(W')^2\subset W$, and $\pi(U)\subset \pi(V\cap W')$. Then $W'\cap U\subset O$.
\end{lemma}

The next theorem is the main result of this section.

\begin{theorem}\label{Th:Ext}
Let $K$ be a compact invariant subgroup of a topological group $G$. If the quotient group $G/K$ 
has a strong $\sigma$-lattice of open homomorphisms onto groups with a countable network (base), then $G$ also has a strong $\sigma$-lattice of open homomorphisms onto groups 
with a countable network (base). 
\end{theorem}

\begin{proof}
We consider only the case when the group $H=G/K$ has a strong $\sigma$-lattice of open homomorphisms onto groups with a countable network, leaving to the reader simple modifications 
of the argument in the case of groups with a countable base. 

Denote by $\mathcal{M}$ a strong $\sigma$-lattice of open homomorphisms of $H$ onto groups 
with a countable network. Let
$$
\mathcal{L}=\{\ker{g}: g\in\mathcal{M}\}.
$$
Since $g(H)$ is topological group with a countable network for each $g\in\mathcal{M}$, every 
element of $\mathcal{L}$ is an invariant admissible subgroup of $G/K$. Let $\pi$ be the quotient homomorphism of $G$ onto $H$. Let also $\mathcal{A}$ be the family of invariant admissible subgroups of $G$. Note that by Lemma~\ref{Le:Adm}, $\pi(N)$ is an admissible (and invariant) subgroup of $H$, for each $N\in\mathcal{A}$.\smallskip 

\noindent \textbf{Claim~A.} \textit{For every $P\in\mathcal{A}$ and every $L\in\mathcal{L}$ with $L\subset \pi(P)$, there exists $N\in\mathcal{A}$ such that $N\subset P$ and $\pi(N)=L$.}\smallskip 

Indeed, it suffices to put $N=P\cap \pi^{-1}(L)$. Then $N$ is an invariant admissible subgroup of $G$ as the intersection of the invariant admissible subgroups $P$ and $\pi^{-1}(L)$ of $G$, so $N\in\mathcal{A}$. It is also clear that $\pi(N)=\pi(P\cap \pi^{-1}(L))=\pi(P)\cap L=L$. This proves our claim.\smallskip 

We define a subfamily $\mathcal{A}^*$ of $\mathcal{A}$ by letting
$$
\mathcal{A}^* = \{N\in\mathcal{A}: \pi(N)\in\mathcal{L}\}.
$$
For every $N\in\mathcal{A}^*$, let $\varphi_N\colon G\to G/N$ be the quotient homomorphism.
Let us show that $\mathcal{F}=\{\varphi_N: N\in\mathcal{A}^*\}$ is a strong $\sigma$-lattice of open homomorphisms of $G$ onto topological groups with a countable network. This requires several steps.

Every space with a countable network is separable and, hence, weakly Lindel\"of. So by 
Lemma~\ref{Le:wL}, the group $H=G/K$ is weakly Lindel\"of. Applying Lemma~\ref{Le:wL2} 
we conclude that the group $G$ is also weakly Lindel\"of. Hence $G$ is $\omega$-narrow 
(see \cite[Corollary~5.2.9]{AT}).\smallskip

Step~I. The group $G/N$ has a countable network, for each $N\in\mathcal{A}^*$. Indeed, 
take an arbitrary element $N\in\mathcal{A^*}$. Since the subgroup $K$ of $G$ is compact and invariant, $NK$ is a closed invariant subgroup of $G$. Let $p\colon G\to G/NK$ be the quotient 
homomorphism. Since $N\subset NK$, there exists a continuous homomorphism $h\colon G/N\to 
G/NK$ satisfying $p=h\circ\varphi_N$. It is clear that $\varphi_N(K)$ is the kernel of $h$. We claim 
that the compact group $\varphi_N(K)$ has a countable base. First, the quotient group $G/N$ has
countable pseudocharacter because $N$ is an admissible subgroup of $G$ \cite[Lemma~5.5.2\,a)]{AT}. Second, every compact subset of a topological group of countable pseudocharacter has a countable base by Lemma~\ref{Le:CM}, whence our claim follows. Thus the group $G/N$ contains 
the closed subgroup $\varphi_N(K)$ with a countable base such that the group $G/NK\cong (G/N)/\varphi_N(K)\cong (G/K)/\pi(N)$ has a countable network (notice that $\pi(N)\in\mathcal{L}$ since $N\in\mathcal{A}^*$). It now follows from \cite{Usp} (the result attributed to M.~Choban) that $G/N$ has 
a countable network.\smallskip

Step~II. The family $\mathcal{F}$ generates the original topology of  $G$. Since the group 
$G$ is $\omega$-narrow, it follows from \cite[Corollary~3.4.19]{AT} that every neighborhood of 
the identity in $G$ contains an invariant admissible subgroup, i.e.~an element of $\mathcal{A}$. 
We need a slightly stronger property of the family $\mathcal{L}$:
\smallskip

\textbf{Claim~B.} \textit{Every $G_\delta$-set $P$ in $H$ with $e_H\in P$ contains an element of 
$\mathcal{L}$, where $e_H$ is the identity of $H$.}\smallskip

Indeed, take a countable family $\{V_n: n\in\omega\}$ of open neighborhoods of $e_H$ in $H$ such that $P=\bigcap_{n\in\omega} V_n$. Since $\mathcal{M}$ is a strong $\sigma$-lattice for $H$, we can find, for every $n\in\omega$, an element $L_n\in\mathcal{L}$ such that $L_n\subset V_n$. Take
$g_n\in\mathcal{M}$ with $\ker g_n=L_n$, where $n\in\omega$. Making use of the fact that 
$\mathcal{M}$ is a strong $\sigma$-lattice for $H$ once again, we choose $g\in\mathcal{M}$ with
$g\prec g_n$ for each $n\in\omega$. Then $L=\ker{g}\in\mathcal{L}$ and $L\subset \bigcap_{n\in\omega} L_n\subset \bigcap_{n\in\omega} V_n = P$. This proves our claim.\smallskip

Let $U$ and $V$ be open neighborhoods of the identity in $G$ such that $V^2\subset U$. It suffices to show that $V$ contains an element $N$ of the family $\mathcal{A}^*$\,---\,then the open neighborhood 
$\varphi_N(V)$ of the identity in $\varphi_N(G)$ satisfies $\varphi_N^{-1}\varphi_N(V) =
VN\subset V^2\subset U$. Hence we take an arbitrary element $P\in\mathcal{A}$ with $P\subset V$.
Then $\pi(P)$ is an admissible subgroup of $H$ and, hence, a $G_\delta$-set in $H$. By Claim~B, there exists an element $L\in\mathcal{L}$ with $L\subset \pi(P)$, so we apply Claim~A to find $N\in\mathcal{A}$ such that $N\subset P$ and $\pi(N)=L$. Then $N\in\mathcal{A}^*$ and clearly $N\subset P\subset V$. We have thus proved that the family $\mathcal{F}$ generates the original topology of $G$. \smallskip

Step~III. The family $\mathcal{F}$ is a strong $\sigma$-lattice for $G$. First, take an arbitrary decreasing sequence $N_0\supset N_1\supset\cdots\supset N_i\supset\cdots$ of elements of 
$\mathcal{A}^*$. We claim that $N=\bigcap_{i\in\omega} N_i$ is an element of $\mathcal{A}^*$
as well. Indeed, it is clear that $N\in\mathcal{A}$. It follows from the definition of the family $\mathcal{A}^*$ that $L_i=\pi(N_i)\in\mathcal{L}$, and clearly $L_{i+1}\subset L_i$, for each $i\in\omega$. Since  
$\mathcal{M}$ is a strong $\sigma$-lattice for $H$, we see that $L=\bigcap_{i\in\omega} L_i$ is in 
$\mathcal{L}$. Further, the compactness of the subgroup $K$ of $G$ implies that $\pi(N)=L$. This proves that $N\in\mathcal{A}^*$.

Let $\varphi\colon G\to G/N$ and $\varphi_i\colon G\to G/N_i$ be quotient homomorphisms, 
where $i\in\omega$. For every $i\in\omega$, there exists an open continuous homomorphism 
$p_i\colon G/N\to G/N_i$ satisfying $\varphi_i=p_i\circ\varphi$. For integers $i,j$ with $0\leq i<j$,
let also $p_{j,i}\colon G/N_j\to G/N_i$ be an open continuous homomorphism satisfying $p_i=p_{j,i}
\circ p_j$. Similarly, for every $i\in\omega$, let $h_i\colon G\to G/N_iK$ and $\psi_i\colon H\to H/L_i
\cong G/N_iK$ be quotient homomorphisms satisfying $h_i=\psi_i\circ\pi$. Let also $\varpi\colon G/N\to 
H/L$ and $\varpi_i\colon G/N_i\to H/L_i$ be continuous homomorphisms satisfying $\varpi\circ\varphi= \psi\circ\pi$ and $\varpi_i\circ\varphi_i=\psi_i\circ\pi$, where $\psi\colon H\to H/L$ is the quotient homomorphism. Denote by $q_i$ a continuous homomorphism of $H/L$ to $H/L_i$ satisfying $\psi_i=q_i\circ\psi$.  
\[
\xymatrix{G \ar@{>}[r]^\pi \ar@{>}[d]_\varphi \ar@/_2.5pc/[dd]_{\varphi_i}  & H \ar@/^2.5pc/[dd]^{\psi_i}
\ar@{>}[d]^\psi\\
G/N \ar@{>}[r]^\varpi \ar@{>}[d]_{p_i}  &  H/L \ar@{>}[d]^{q_i}\\
G/N_i \ar@{>}[r]^{\varpi_i} & H/L_i}
\]
It is clear that all homomorphisms in the above diagram are open. 

To finish the proof of the statement in Step~III (and of the theorem) it suffices to verify that 
the limit of the inverse sequence 
$$
\{G/N_k,\, p_{l,k}: k,l\in\omega,\ k<l\}
$$ 
is topologically isomorphic to the group $G/N$. Let $O$ be an arbitrary open symmetric neighborhood of the identity $e_*$ in $G/N$. There exists an open symmetric neighborhood $V$ of $e_*$ in $G/N$ such that $V^2\subset O$. It is easy to find an integer $i\in\omega$ and an open symmetric neighborhood $W_i$ of the identity in $G/N_i$ such that $p_i^{-1}(W_i)\cap p(K)\subset V$. Indeed, 
$p(K)$ is a compact subspace of $G/N$. Let $y$ be an arbitrary element of $p(K)$ distinct from 
$e_*$. Choose $x\in K$ with $p(x)=y$. There exists $i\in\omega$ such that $p_i(x)$ is distinct from 
the identity of $G/N_i$\,---\,otherwise $x\in\bigcap_{i\in\omega} N_i=N$ and, hence, $y=p(x)=e_*$, which contradicts our choice of $y$. In other words, we see that the family $\{p_k: k\in\omega\}$ of 
continuous homomorphisms of $G/N$ separates points of the compact group $p(K)$. Therefore 
this family generates the same topology on $p(K)$ that $p(K)$ inherits from the group $G/N$. 
In particular, this implies the existence of the required $i\in\omega$ and $W_i\subset G/N_i$. 
We choose an open symmetric neighborhood $W_{i+1}$ of the identity in $G/N_i$ such that
$W_{i+1}^2\subset W_i$.

Further, choose $k\in\omega$ with $k>i$ and an open symmetric neighborhood $U_k$ of the identity in
the group $H/L_k$ such that $q_k^{-1}(U_k) \subset \varpi(V\cap p_i^{-1}(W_{i+1}))$. This is possible 
since $\mathcal{M}$ is a $\sigma$-lattice for $H$ and $L=\bigcap_{n\in\omega} L_n$, where 
$L_n\in\mathcal{L}$ for each $n\in\omega$. 

We claim that the open neighborhood $W^*=p_{k,i}^{-1}(W_{i+1})\cap \varpi_k^{-1}(U_k)$ of the identity
in $G/N_k$ satisfies $p_k^{-1}(W^*)\subset O$, which shows that the family $\{p_n: n\in\omega\}$
of continuous homomorphisms of the group $G/N$ generates the original quotient topology of
$G/N$. Indeed, it follows from the commutativity of the above diagram that $p_k^{-1}(W^*)=
p_i^{-1}(W_{i+1})\cap \varpi^{-1}(q_k^{-1}(U_k))$. According to our choice of the sets $W_i$, 
$W_{i+1}$ and $U_k$ we have $W_{i+1}^2\subset W_i$, $p_i^{-1}(W_i)\cap p(K)\subset V$, and $q_k^{-1}(U_k)\subset \varpi(V\cap p_i^{-1}(W_{i+1}))$. Therefore we can apply Lemma~\ref{Le:AT}
with $U=\varpi^{-1}(q_k^{-1}(U_k))$, $W=p_i^{-1}(W_i)$, and $W'=p_i^{-1}(W_{i+1})$ to conclude that
$$
p_k^{-1}(W^*) =  p_k^{-1}\big(p_{k,i}^{-1}(W_{i+1}))\cap p_k^{-1}(\varpi_k^{-1}(U_k)\big) =
p_i^{-1}(W_{i+1})\cap \varpi^{-1}(q_k^{-1}(U_k))\subset O.
$$

Finally it remains to verify that the group $G/N$ and the limit of the inverse sequence 
$S=\{G/N_k,\, p_{l,k}: k,l\in\omega,\ k<l\}$ coincide as sets. Since the family $\{p_k: k\in\omega\}$
separates points of $G/N$, we see that the canonical mapping of $G/N$ to the limit of $S$ is
one-to-one. Hence it suffices to show that this mapping is onto. 

Let us note that for every $i\in\omega$, the mappings in the equality $q_i\circ\varpi=\varpi_i\circ{p_i}$
(see the above diagram) are \emph{bi-commutative} in the sense that $q_i^{-1}(\varpi_i(z))=
\varpi(p_i^{-1}(z))$ for each $z\in G/N_i$. Indeed, since all the mapping here are homomorphisms,
we only need to verify the latter equality in the case when $z$ is the identity $e_i$ of $G/N_i$. 
Direct calculations show that $q_i^{-1}(\varpi_i(e_i))=\ker q_i=\psi(L_i)$ and $\varpi(p_i^{-1}(e_i))=\varpi(\ker p_i)=\varpi(\varphi(N_i))=\psi(\pi(N_i))=\psi(L_i)$, which gives the required equality.

Let $\{x_k: k\in\omega\}$ be a sequence of points such that $x_k\in G/N_k$ and $p_{k+1,k}
(x_{k+1})=x_k$ for each $k\in\omega$. Our aim is to find $x\in G/N$ satisfying $p_k(x)=x_k$
for each $k$. Let $y_k=\varpi_k(x_k)$. Then $q_{k+1}(y_{k+1})=y_k$ for all $k$, so our choice
of the decreasing sequence $\psi_0\succ \cdots\succ \psi_n\succ  \psi_{n+1} \cdots\succ \psi$ in 
the $\sigma$-lattice $\mathcal{M}$ implies that there exists $y\in H/L$ such that $q_k(y)=y_k$ for 
each $k\in\omega$. Making use of the bi-commutativity of the diagrams $q_k\circ\varpi=
\varpi_k\circ{p_k}$, we have that $y\in q_k^{-1}(y_k)=q_k^{-1}(\varpi_k(x_k))=\varpi(p_k^{-1}(x_k))$,
for each $k\in\omega$. Hence $p_k^{-1}(x_k)\cap\varpi^{-1}(y)\neq\emptyset$ for each $k$.
Notice that $\varpi$ is a perfect homomorphism since its kernel is the compact group $\varphi(K)$. 
Hence the inclusions $p_{k+1}^{-1}(x_{k+1})\subset p_k^{-1}(x_k)$ imply that $\varpi^{-1}(y)\cap
\bigcap_{k\in\omega} p_k^{-1}(x_k)\neq\emptyset$. Take an arbitrary element $x\in \bigcap_{k\in\omega} p_k^{-1}(x_k)$. Then $p_k(x)=x_k$ for each $k\in\omega$, so the limit of the inverse
sequence $S$ and the group $G/N$ coincide as sets. Since in addition the family 
$\{p_k: k\in\omega\}$ generates the quotient topology of the group $G/N$, we infer that 
$G/N$ and the limit of $S$ are topologically isomorphic groups. So the family $\mathcal{F}$ 
is a strong $\sigma$-lattice for $G$.

It is worth noting that in the case when $\mathcal{M}$ is a strong $\sigma$-lattice for $G/K$ of
open homomorphisms onto groups with countable base, our argument requires the fact that an
extension of a second countable group by another second countable group is again second
countable (see \cite[Corollary~3.3.21]{AT}).
\end{proof}

Taking $G=K$ in Theorem~\ref{Th:Ext} (and arguing as in its proof), we obtain the following
fact:

\begin{corollary}\label{Cor:Com}
Let $\mathcal{N}$ be the family of closed invariant subgroups of type $G_\delta$ in a compact topological group $C$. For every $N\in\mathcal{N}$, let $\pi_N\colon C\to C/N$ be the quotient
homomorphism. Then the family $\{\pi_N: N\in\mathcal{N}\}$ is a strong $\sigma$-lattice of open
homomorphisms of $C$ onto groups with a countable base.
\end{corollary}

The above corollary admits a slight generalization.

\begin{corollary}\label{Cor:Pcom}
Let $\mathcal{A}$ be the family of admissible invariant subgroups of a pseudocompact topological group $P$. For every $N\in\mathcal{A}$, let $\pi_N\colon C\to P/N$ be the quotient homomorphism. Then the family $\{\pi_N: N\in\mathcal{A}\}$ is a strong $\sigma$-lattice of open homomorphisms of 
$P$ onto groups with a countable base.
\end{corollary}

\begin{proof}
Take an arbitrary element $N\in\mathcal{A}$. Then the identity of the quotient group $P/N$ is a
$G_\delta$-set in $C/N$ \cite[Lemma~5.5.2]{AT}, so $C/N$ is a compact metrizable group, by \cite[Lemma~3.1]{CS}. Let $\varrho{C}$ be the Ra\u{\i}kov completion of $C$. Then $\varrho{C}$ is a compact topological group containing $C$ as a dense subgroup. In fact, $C$ meets every non-empty $G_\delta$-set in $\varrho{C}$ (see \cite[Theorem~1.2]{CR}). For every $N\in\mathcal{N}$, the homomorphism $\pi_N\colon C\to C/N$ extends to a continuous homomorphism $p_N\colon\varrho{C}\to C/N$ \cite[Corollary~3.6.17]{AT}. Therefore we have the equality $p_N(C)=p_N(\varrho{C})=C/N$. Since the groups $\varrho{C}$ and $C/N$ are compact, the homomorphism $p_N$ is open. To finish 
our argument it suffices to apply Corollary~\ref{Cor:Com}.
\end{proof}

\begin{remark}
{\rm The reader surely noted the difference in our definitions of the $\sigma$-lattices $\mathcal{N}$ 
and $\mathcal{A}$ in Corollaries~\ref{Cor:Com} and~\ref{Cor:Pcom}, respectively. The elements of
$\mathcal{N}$ are closed subgroups \emph{of type $G_\delta$} in the compact group $C$, while
the elements of $\mathcal{A}$ are \emph{admissible} subgroups of the pseudocompact group $P$. 
Clearly every admissible subgroup of a topological group is of type $G_\delta$ in the group, while
the converse is false in the general case. However, the main reason for requiring the elements of 
$N\in\mathcal{A}$ to be admissible is to guarantee that the quotient group $P/N$ have countable 
pseudocharacter. In fact, one can also use the family $\mathcal{N}$ in the proof of Corollary~\ref{Cor:Pcom}, but this requires an extra argument.} 
\end{remark}

\begin{problem}\label{Prob:Inv}
Let $K$ be a compact invariant subgroup of a topological group $G$ such that the quotient group
$G/K$ is $\omega$-balanced. Is the group $G$ $\omega$-balanced?
\end{problem}

\begin{problem}\label{Prob:Metr}
Let $K$ be a compact invariant subgroup of a topological group $G$ such that the quotient group
$G/K$ has a strong $\sigma$-lattice of open homomorphisms onto metrizable topological groups.
Does $G$ have a strong $\sigma$-lattice of open homomorphisms onto metrizable topological 
groups?
\end{problem}

The affirmative answer to Problem~\ref{Prob:Inv} would imply that the answer to 
Problem~\ref{Prob:Metr} is \lq\lq{yes\rq\rq}.

\begin{problem}\label{Prob:Quo}
Let a topological group $G$ have a strong $\sigma$-lattice of open homomorphisms onto 
metrizable groups (groups with a countable network or a countable base) and $K$ be a compact invariant subgroup of $G$. Does the group $G/K$ have a strong $\sigma$-lattice of open homomorphisms onto metrizable groups (groups with a countable network or a countable base)?
\end{problem}

%%%%%%%%%%%%%%%%%%%%%%%%%%%%%%%%
\section{Extensions of pro-Lie groups}\label{Ext_pro-Lie}
%%%%%%%%%%%%%%%%%%%%%%%%%%%%%%%%
The class of groups having the properties described in the conclusion of Theorem~\ref{Th:1} is even wider than continuous homomorphic images of almost connected pro-Lie groups. It turns out that an extension of an almost connected pro-Lie group by a compact group is in this class. This fact follows from a more general result given below. 

\begin{theorem}\label{Th:2}
Let $G$ be a topological group and $K$ be a compact invariant subgroup of $G$ such that the quotient group $G/K$ is homeomorphic as a space to the product $C\times \prod_{i\in I} H_i$, where $C$ is a compact group and, for every $i\in I$, $H_i$ is a topological group with a countable network. Then the group $G$ is $\R$-factorizable, $\omega$-cellular, and the closure of every $G_{\delta,\Sigma}$-set in $G$ is a zero-set, so $G$ is an Efimov space.
\end{theorem}

\begin{proof}
First we claim that the product group $\Pi=C\times \prod_{i\in I} H_i$ has a strong 
$\sigma$-lattice of open homomorphisms onto topological groups with a countable network. 
Indeed, Corollary~\ref{Cor:Pcom} guarantees that the group $C$ has a strong $\sigma$-lattice 
of open homomorphisms onto groups with a countable base. Let $\mathcal{C}$ be such a lattice 
for $C$. For every countable set $J\subset I$, let $p_J$ be the projection of $\Pi$ onto 
$\prod_{i\in J} H_i$. It is easy to see that the family 
$$
\mathcal{M}= \{f\times p_J: f\in\mathcal{C},\ \emptyset\neq J\subset I,\ |J|\leq\omega \}
$$
is a strong $\sigma$-lattice for $\Pi$ consisting of open homomorphisms onto groups with a 
countable network. Hence the space $\Pi$ is weakly Lindel\"of by Lemma~\ref{Le:wL}. Applying Lemma~\ref{Le:wL2} we infer that the group $G$ is also weakly Lindel\"of. Hence the group 
$G/K$ is $\omega$-narrow by \cite[Corollary~5.2.9]{AT}. 

Denote by $\mathcal{M}^*$ the weak $\sigma$-lattice for the group $G/K$ which consists 
of all continuous open homomorphisms of $G/K$ onto topological groups of countable pseudocharacter. Equivalently, one can define $\mathcal{M}^*$ as the family of all quotient homomorphisms $\varphi\colon G/K\to (G/K)/L$, where $L$ is an admissible invariant subgroup 
of $G/K$. Since the spaces $G/K$ and $\Pi$ are homeomorphic, it follows from Lemma~\ref{Le:6} 
that $\mathcal{M}^*$ contains a cofinal subfamily, say, $\mathcal{L}^*$ which is a strong 
$\sigma$-lattice for $G/K$. 

Since the subgroup $K$ of $G$ is compact and invariant, we can apply Theorem~\ref{Th:Ext} to
conclude that the group $G$ itself has a strong $\sigma$-lattice of open homomorphisms onto
groups with a countable network, say, $\mathcal{F}$. The family $\mathcal{F}$ has the factorization property according to Lemma~\ref{Le:5.0}. Since every topological group with a countable network 
is $\R$-factorizable \cite[Corollary~8.1.7]{AT}, we see that the group $G$ is $\R$-factorizable as well
(one can apply \cite[Lemma~8.1.11]{AT} here). To complete the argument one can follow the patterns in the proofs of Theorem~\ref{Th:1} and Theorem~\ref{Th:1y}.
\end{proof}

Since every almost connected pro-Lie group is \emph{homeomorphic} to the product $C\times
\R^\kappa$, where $C$ is a compact group and $\kappa$ is a cardinal, the next fact is immediate
from Theorem~\ref{Th:2}.

\begin{corollary}\label{Cor:EXT}
If a topological group $G$ contains a compact invariant subgroup $K$ such that the quotient group $G/K$ is homeomorphic to an almost connected pro-Lie group, then $G$ is $\R$-factorizable, 
$\omega$-cellular, and $G$ is an Efimov space.
\end{corollary}

The authors are grateful to K.\,H.~Hofmann and S.\,A.~Morris for providing us with an argument 
for the proof of the following lemma.

\begin{lemma}\label{Le:HMo}
Let $G$ be a pro-Lie group and $K$ be a compact invariant subgroup of $G$ such that the 
quotient group $H=G/K$ is a connected pro-Lie group. Then the group $G$ is almost connected. 
Furthermore, if $f\colon G\to G/K$ is the quotient homomorphism and $G_0$ is the connected
component of $G$, then $f(G_0)=H$. 
\end{lemma}

\begin{proof}
Clearly $f$ is a closed mapping. The connected component $G_0$ of $G$ is a closed 
invariant subgroup of $G$. According to \cite[Corollary~4.22\,(iii)]{PROBOOK}, $f(G_0)$ 
is dense in the connected group $G/K$. Since $G_0$ is closed in $G$ and the mapping $f$ 
is closed, we see that $f(G_0)=G/K$. Hence $G_0K=G$ and the quotient group $G/G_0$ is 
compact. This implies that the group $G$ is almost connected.
\end{proof}

In the following theorem we weaken \lq{connected\rq} to \lq{almost connected\rq} 
in the assumptions of Lemma~\ref{Le:HMo}.

\begin{proposition}\label{prop_HM}
Let $G$ be a pro-Lie group and $K$ be a compact invariant subgroup of $G$ such that the 
quotient group $G/K$ is an almost connected pro-Lie group. Then $G$ is almost connected. 
\end{proposition}

\begin{proof}
Let  $f\colon G\to G/K$ be the quotient homomorphism. Denote by $G_0$ and $H_0$  
the connected components of $G$ and the quotient group $H=G/K$, respectively. Then 
$G_0$ and $H_0$ are closed invariant subgroups and $G^*=f^{-1}(H_0)$ is also a closed 
invariant subgroup of $G$. Hence $G^*$ is a pro-Lie group as a closed subgroup of $G$ 
\cite[Theorem 3.35]{PROBOOK}. Since $G^*/K\cong H_0$ is a closed subgroup of $H$, 
we conclude that $G^*/K$ is a connected pro-Lie group. It follows from $G_0\leq G^*\leq G$ 
that the connected component of $G^*$ is $G_0$. Hence Lemma~\ref{Le:HMo} implies that 
$f(G_0)=H_0$ and that $G^*=f^{-1}(H_0)=G_0K$ is an almost connected pro-Lie group. 
Clearly the quotient group $L=G^*/G_0$ is compact. Denote by $\pi$ the quotient 
homomorphism of $G$ onto $G/G_0$. Then $L=G^*/G_0\cong \pi(K)$ is a compact 
subgroup of $G/G_0$. To finish the proof it suffices to verify that the group $G/G_0$ 
is compact, i.e.~$G$ is almost connected.

Let $p\colon G\to H^*$ be the quotient homomorphism, where $H^*=G/G^*$. Then 
we can represent $p$ in the form $p=q\circ \pi$, where $q\colon G/G_0\to (G/G_0)/L$ 
is the quotient homomorphism. Clearly the kernel of $q$ is the compact subgroup 
$L$ of $G/G_0$. Further, the quotient group $G/G^*$ is topologically isomorphic to the
group $H/f(G_0)=H/H_0$ which is compact since $H$ is almost connected. Hence $q$ 
is a quotient homomorphism of $G/G_0$ onto the compact group $H^*/L\cong H/H_0$
and the kernel of $q$ is compact. We conclude therefore that the group $G/G_0$ is also 
compact by the well-known fact that compactness is a three space property in topological 
groups \cite[Corollary 1.5.8]{AT}. This completes the proof of the theorem.
\end{proof}

We don't know if Proposition~\ref{prop_HM} can be generalized assuming only that the quotient 
group $G/K$ is \emph{homeomorphic} to an almost connected pro-Lie group.

It is also an open question whether Proposition~\ref{prop_HM} remains valid without assuming 
\emph{a priori} that $G$ is a pro-Lie group.

%%%%%%%%%%%%%%%%%%%%%%%%%%%%%%%%%%%%%%%
\section{Convergence properties of pro-Lie groups}\label{Sec:CS}
%%%%%%%%%%%%%%%%%%%%%%%%%%%%%%%%%%%%%%%
Our aim in this section is to show in Theorem~\ref{Th:4} that the closure and sequential 
closure of an arbitrary $G_{\delta,\Sigma}$-subset of an almost connected pro-Lie group 
coincide. As usual, we say that a subset $Y$ of a space $X$ is a $G_{\delta,\Sigma}$-set 
in $X$ provided that $Y$ is the union of a family of $G_\delta$-sets in $X$.

In the sequel we will use a result proved in \cite[Theorem~2.8]{BCDT} in the case of Abelian
topological groups. However, one can repeat the corresponding arguments in \cite{BCDT}
without the use of commutativity of the groups involved there, thus obtaining the following 
fact:

\begin{theorem}\label{Th:BCDT}
Let $K$ be a compact invariant subgroup of a topological group $X$ and $p\colon X\to X/K$ the quotient homomorphism. If $Y$ is a zero-dimensional compact subspace of $X/K$, then there 
exists a continuous mapping $s\colon Y\to X$ satisfying $p\circ{s} =Id_Y$. 
\end{theorem}

\begin{theorem}\label{Th:4}
Let $H$ be an almost connected pro-Lie group. Then, for every $G_{\delta,\Sigma}$-set $P$ 
in $H$ and every point $x\in\overline{P}$, the set $P$ contains a sequence converging to $x$. 
In other words, the closure of $P$ and the sequential closure of $P$ in $H$ coincide.
\end{theorem}

\begin{proof}
By \cite[Corollary~8.9]{HM2}, the group $H$ is homeomorphic to the product $\R^\kappa\times K$, where $\R$ is the real line with the usual topology, $\kappa$ is a cardinal, and $K$ is a compact topological group. Therefore, it suffices to deduce the conclusion of the theorem in the case when 
$H$ is topologically isomorphic to the product group $\R^\kappa\times K$. For simplicity we identify 
$H$ with $\R^\kappa\times K$. 

Let $P$ be a $G_{\delta,\Sigma}$-set in $H$ and $x\in\overline{P}\setminus P$ be an arbitrary point. 
Take a family $\gamma$ of $G_\delta$-sets in $H$ such that $P=\bigcup\gamma$. Refining the elements of $\gamma$, if necessary, we can assume that every element of $\gamma$ is a zero-set
in $H$. It follows from Corollary~\ref{Cor:AlC} that $cel_\omega(H)\leq\omega$, so $\gamma$
contains a countable subfamily whose union is dense in the union of $\gamma$. This in turn enables
us to assume that the family $\gamma$ is countable. 

Let $\mathcal{C}$ be a strong $\sigma$-lattice of continuous open homomorphisms of the compact 
group $K$ onto second countable topological groups. For every non-empty set $J\subset\kappa$,
we denote by $p_J$ the projection of $\R^\kappa$ onto $\R^J$. As in the proof of Theorem~\ref{Th:2}
one can verify that the family
$$
\mathcal{M} = \{p_J\times f: f\in\mathcal{C},\ \emptyset\neq J\subset\kappa,\ |J|\leq\omega\}
$$
is a strong $\sigma$-lattice for $H$ and, therefore, $\mathcal{M}$ has the factorization property. 
Since every element of $\gamma$ is a zero-set in $H$, we conclude that there exists an element 
$\varphi\in\mathcal{M}$ such that $F=\varphi^{-1}\varphi(F)$, for each $F\in\gamma$. Let
$\varphi=p_J\times f$, where $J$ is a non-empty countable subset of $\kappa$ and $f\in\mathcal{C}$. Then the groups $C=f(K)$ and $\R^J\times C$ are second countable and $\varphi(x)\in\overline{\varphi(\bigcup\gamma)}$. Hence we can find a sequence $\{z_n: n\in\omega\}\subset \varphi(\bigcup\gamma)$ converging to the point $\varphi(x)$. 

For every $n\in\omega$, let $z_n=(u_n,c_n)$, where $u_n\in \R^J$ and $c_n\in C$. 
Let also $x=(x^*,a)$, where $x^*\in\R^\kappa$ and $a\in K$. Then $\varphi(x)=(y^*,b)$, where $y^*=p_J(x^*)\in \R^J$ and $b=f(a)\in C$. It is clear that the sequences $\{u_n: n\in\omega\}$ and 
$\{c_n: n\in\omega\}$ converge to $y^*$ and $b$, respectively. For every $n\in\omega$, we define 
a point $y_n\in\R^\kappa$ by the rule $y_n(\alpha)=u_n(\alpha)$ if $\alpha\in J$ and 
$y_n(\alpha)=x^*(\alpha)$ if $\alpha\in\kappa\setminus J$. Then the sequence 
$\{y_n: n\in\omega\}$ converges to $x^*$, so the sequence $\{(y_n,c_n): n\in\omega\}$ 
converges to $(x^*,b)$. 

Let $j$ be the identity mapping of $\R^\kappa$ onto itself and $e^*$ the identity element 
of $\R^\kappa$. Then $\pi=j\times f$ is a continuous homomorphism of $H$ onto the group 
$\R^\kappa\times C$ which satisfies $\ker\pi\subset \{e^*\}\times K$ and $\pi\prec \varphi$. 
Since $f$ is open, the product homomorphism $\pi$ is open as well. It follows from 
$\pi\prec \varphi$ that $F=\pi^{-1}\pi(F)$, for each $F\in\gamma$. We claim that 
$(y_n,c_n)\in\pi(\bigcup\gamma)$, for each $n\in\omega$. Indeed, the 
homomorphism $\varphi$ can be represented as the composition $\psi\circ \pi$, where 
$\psi\colon\R^\kappa\times C\to \R^J\times C$ is a homomorphism defined by $\psi(u,c)=
(p_J(u),c)$ for all $u\in \R^\kappa$ and $c\in C$. Take an arbitrary integer $n\in\omega$. 
It follows from our definition of the element $y_n$ that $\psi(y_n,c_n)=(u_n,c_n)=z_n$, so 
there exists $F\in\gamma$ such that $z_n\in\varphi(F)$. Since $F=\varphi^{-1}\varphi(F)$ 
and $\varphi=\psi\circ\pi$, we see that $(y_n,c_n)\in\pi(F)$. This proves our claim. 

Let $t_n=(y_n,c_n)$, for each $n\in\omega$. The sequence $\{t_n: n\in\omega\}$ converges
to the element $t=(x^*,b)=\pi(x)$ of $\R^\kappa\times C$. Hence $B=\{t\}\cup\{t_n: n\in\omega\}$ 
is a compact zero-dimensional subset of $\R^\kappa\times C$. Since the kernel of $\pi$, say, $N$ is contained in $\{e^*\}\times K$, it is clear that $N$ is a compact subgroup of $H$. So we can apply Theorem~\ref{Th:BCDT} to find a continuous mapping $s\colon B\to H$ satisfying $\pi\circ{s}=Id_B$. Let $x_n=s(t_n)$, for each $n\in\omega$. Then the sequence $\{x_n: n\in\omega\}$ converges to the element $h=s(t)\in H$, where $\pi(h)=\pi(s(t))=t$. It is also clear that $\pi(x_n)=\pi(s(t_n))=t_n$, for 
each $n\in\omega$. Since $F=\pi^{-1}\pi(F)$, for each $F\in\gamma$, we see that 
$\{x_n: n\in\omega\}\subset\bigcup\gamma$. 

If $h=x$, we are done. So let $h\neq x$. It follows from $\pi(x)=t=\pi(h)$ that $g=h^{-1}x\in N=\ker\pi$. 
For every $n\in\omega$, we put $x_n'=x_n\cdot g$. Then the sequence $\{x_n': n\in\omega\}$ converges to $x=h\cdot g$. Let $n\in\omega$ be arbitrary and take $F\in\gamma$ such that 
$x_n\in F$. Then $x_n'=x_n\cdot g\in FN=\pi^{-1}\pi(F)=F$, whence it follows that the sequence 
$\{x_n': n\in\omega\}$ is covered by $\gamma$. This completes the proof of the theorem.
\end{proof}

%\section*{Acknowledgements} 
%The second author gratefully acknowledges the financial support received from the  Center for Advanced Studies in Mathematics of the Ben Gurion University of the Negev during his visit in October, 2015. He was also supported by the Consejo Nacional de Ciencia y Tecnolog\'{\i}a (CONACyT) of Mexico, grant number 265992 (Estancias Sab\'aticas en el Extranjero).

%%%%%%%%%%%%%%%%%%%%%%%%%%%%%%%%%

\end{document}